\documentclass[article]{siamart1116}

\usepackage{amsmath}
\usepackage{latexsym}
\usepackage{amssymb}
\usepackage{lipsum}
\usepackage{amsfonts}
\usepackage{graphicx}
\usepackage{epstopdf}
\usepackage{algorithmic}
\ifpdf
  \DeclareGraphicsExtensions{.eps,.pdf,.png,.jpg}
\else
  \DeclareGraphicsExtensions{.eps}
\fi

\numberwithin{theorem}{section}

\newcommand{\TheTitle}{Extragradient method with variance reduction for stochastic variational inequalities} 
\newcommand{\TheAuthors}{A. Iusem, A. Jofr\'e, R. I. Oliveira, and P. Thompson}

\headers{Extragradient method with variance reduction}{\TheAuthors}

\title{{\TheTitle}\thanks{Submitted to the editors DATE.
}}

\author{
Alfredo N. Iusem\thanks{Instituto de Matem\'atica Pura e Aplicada (IMPA), 
Rio de Janeiro, RJ, Brazil.
(\email{iusp@impa.br}).}
\and
Alejandro Jofr\'e\thanks{Centro de Modelamiento Matem\'atico (CMM \& DIM), 
Santiago, Chile.
(\email{ajofre@dim.uchile.cl}).}
\and
Roberto I. Oliveira\thanks{Instituto de Matem\'atica Pura e Aplicada (IMPA), 
Rio de Janeiro, RJ, Brazil.
(\email{rimfo@impa.br}). Roberto I. Oliveira's work was supported by a {\em Bolsa de Produtividade em Pesquisa} from CNPq, Brazil. His work in this article is part of the activities of FAPESP Center for Neuromathematics (grant \#2013/07699-0, FAPESP - S. Paulo Research Foundation).}
\and
Philip Thompson\thanks{Instituto de Matem\'atica Pura e Aplicada (IMPA), 
Rio de Janeiro, RJ, Brazil.
(\email{philip@impa.br}). Philip Thompson's work was supported by a CNPq Doctoral scholarship while he was a PhD student at IMPA with visit appointments at CMM. His work was conducted at IMPA and CMM.}
}

\usepackage{amsopn}

\newcommand{\esp}{\mathbb{E}}
\newcommand{\prob}{\mathbb{P}}
\newcommand{\var}{\mathbb{V}}
\newcommand{\alg}{\mathcal{F}}
\newcommand{\alge}{\mathcal{G}}

\newcommand{\re}{\mathbb{R}}

\newcommand{\qc}{\mathfrak{q}}
\DeclareMathOperator*{\covar}{cov}

\newcommand{\vertiii}[1]{{\left\vert\kern-0.4ex #1 
\kern-0.4ex\right\vert}}
\newcommand{\Lpnorm}[1]{\vertiii{\,#1\,}_{p}}
\newcommand{\Lqnorm}[1]{\vertiii{\,#1\,}_{q}}

\newcommand{\Lqcnorm}[1]{\vertiii{\,#1\,}_{\mathfrak{q}}}	

\DeclareMathOperator*{\dist}{d}
\DeclareMathOperator*{\diam}{diam}
\newcommand{\argmin}{\mbox{argmin}}

\newtheorem{lema}{Lemma}
\newtheorem{thm}{Theorem}

\newtheorem{prop}{Proposition}
\newtheorem{assump}{\textbf{Assumption}}
\newtheorem{algo}{Algorithm}
\newtheorem{rem}{Remark}
\newtheorem{example}{Example}

\ifpdf
\hypersetup{
  pdftitle={\TheTitle},
  pdfauthor={\TheAuthors}
}
\fi

\begin{document}

\maketitle

\begin{abstract}
We propose an extragradient method with stepsizes bounded away from zero for stochastic variational inequalities requiring only pseudo-monotonicity. We provide convergence and complexity analysis, allowing for an unbounded feasible set, unbounded operator, non-uniform variance of the oracle and, also, we do not require any regularization. Alongside the stochastic approximation procedure, we iteratively reduce the variance of the stochastic error. Our method attains the optimal oracle complexity $\mathcal{O}(1/\epsilon^2)$ (up to a logarithmic term) and a faster rate $\mathcal{O}(1/K)$ in terms of the mean (quadratic) natural residual and the D-gap function, where $K$ is the number of iterations required for a given tolerance $\epsilon>0$. Such convergence rate represents an acceleration with respect to the stochastic error. The generated sequence also enjoys a new feature: the sequence is bounded in $L^p$ if the stochastic error has finite $p$-moment. Explicit estimates for the convergence rate, the oracle complexity and the $p$-moments are given depending on problem parameters and distance of the initial iterate to the solution set. Moreover, sharper constants are possible if the variance is uniform over the solution set or the feasible set. Our results provide new classes of stochastic variational inequalities for which a convergence rate of $\mathcal{O}(1/K)$ holds in terms of the mean-squared distance to the solution set. Our analysis includes the distributed solution of pseudo-monotone Cartesian variational inequalities under partial coordination of parameters between users of a network.  
\end{abstract}



\section{Introduction}\label{s1}
The standard (deterministic) variational inequality problem, which we will denote as VI($T,X)$ or simply VI, is defined as follows: 
given a closed and convex set $X\subset\re^n$ and a single-valued operator $T:\mathbb{R}^n\rightarrow\mathbb{R}^n$, 
find $x^*\in X$ such that for all $x\in X$,
\begin{equation}\label{VI}
\langle T(x^*),x-x^*\rangle\ge0.
\end{equation}
We shall denote by $X^*$ the solution set of VI$(T,X)$. The variational inequality problem includes many interesting special 
classes of variational problems 
with applications in economics, game theory and engineering. The basic prototype is smooth convex optimization, 
when $T$ is the gradient of a smooth function.  
Other problems which can be formulated as variational inequalities, include \emph{complementarity 
problems} (when $X=\re^n_+$), \emph{systems of equations} (when $X=\re^n$), \emph{saddle-point problems} and 
many \emph{equilibrium problems}. We refer the reader to Chapter 1 of \cite{facchinei} and \cite{ferris:pang} for an extensive review 
of applications of the VI problem in engineering and economics. The complementarity problem and system of equations are important classes of problems where the feasible set is unbounded.

In the stochastic case, we start with a measurable space $(\Xi,\alge)$, a measurable (random) operator 
$F:\Xi\times\re^n\to\re^n$ and a random variable $\xi:\Omega\rightarrow\Xi$ defined on a probability space 
$(\Omega,\alg,\prob)$ which induces an expectation $\esp$ and a distribution $\prob_\xi$ of $\xi$. 
When no confusion arises, we sometimes use $\xi$ to also denote a random sample  $\xi\in\Xi$. We assume 
that for every $x\in\re^n$, $F(\xi,x):\Omega\rightarrow\re^n$ is an integrable random vector.
The solution criterion analysed in this paper consists
of solving VI($T,X$) as defined by \eqref{VI}, where $T:\re^n\to\re^n$ is the expected value of $F(\xi,\cdot)$, i.e., 
\begin{equation}\label{expected}
T(x)=\esp[F(\xi,x)]=\int_\Omega F(\xi(\omega),x)\dist\prob(\omega),\quad\forall x\in\re^n.
\end{equation} 
For clarity, we state such formulation of the \emph{stochastic variational inequality} problem (SVI) in the following definition:

\quad

\begin{definition}[SVI]\label{SVI.def}
	Assuming that $T:\re^n\rightarrow\re^n$ is given by $T(x)=\esp[F(\xi,x)]$ for all $x\in\re^n$, the \emph{SVI} problem consists of
finding $x^*\in X$, such that $\langle T(x^*),x-x^*\rangle\ge0$ for all $x\in X$.
\end{definition}	

\quad

Such formulation of SVI is often called \emph{expected value} (EV) formulation. The introduction of this formulation goes back to \cite{king:rockafellar,robinson2}, as a natural generalization of stochastic optimization problems (SP). We remark here that different formulations for the stochastic variational inequality problem exist in which the randomness is treated differently. For instance, in the so called \emph{expected residual minimization} (ERM) formulation, one defines a suitable non-negative function $x\mapsto h(\xi,x)$ whose zeros are solutions of VI$(F(\xi,\cdot),X)$. The ERM formulation is then defined as the problem $\min_{x\in X}\esp[h(\xi,x)]$. Variants of the ERM formulation can include random constraints as well. Both EV and ERM formulations have relevance in modelling stochastic equilibrium problems in different settings. See e.g. \cite{ravat:shanbhag, chen&wets}.

Methods for the deterministic VI($T,X$) have been extensively studied (see \cite{facchinei}). If $T$ is fully available 
then SVI can be solved by these methods. As in the case of SP, the SVI in Definition \ref{SVI.def} becomes 
very different from the deterministic setting when $T$ is \emph{not available}. This is often the case in practice due to expensive computation of the expectation in \eqref{expected}, unavailability of $\prob_\xi$ or no closed form for $F(\xi,\cdot)$. This requires sampling the random variable $\xi$ and the use of values of $F(\eta,x)$ given a sample $\eta$ of $\xi$ and a current point $x\in\re^n$ (a procedure often called ``stochastic oracle'' call).
In this context, there are two current methodologies for solving the SVI problem: \emph{sample average approximation} (SAA)  and \emph{stochastic approximation} (SA). In this paper we focus on the SA approach. 
For analysis of the SAA methodology for SP, see e.g., \cite{shapiro} and references therein. For the analysis of the SAA methodology for solving SVIs see e.g. \cite{robinson2, xu:zhang,xu1,xu2}.

We make some remarks regarding the solution of the SVI problem in Definition \ref{SVI.def} when using stochastic approximation (SA). The SA methodology for SP or SVI can be seen as a projection-type method where the exact mean operator $T$ is replaced along the iterations by a 
random sample of $F$. This approach induces a stochastic error $F(\xi,x)-T(x)$ for $x\in X$ in the trajectory of the method. In this solution method the generated sequence $\{x^k\}$ is unavoidably a stochastic process which evolves recursively according to the chosen projection algorithm and the sampling information used in every iteration. As a consequence, asymptotic convergence of the SA method guarantees a solution of Definition \ref{SVI.def} with \emph{total probability}. Precisely, limit points of the sequence $\{x^k\}$ are typically a random variable $x^*$ such that, with \emph{total probability}, $x^*\in X^*$. 

The first analysis of SA methodology for SVI was carried out recently in \cite{Xu}. When $X=\re^n$, Definition \ref{SVI.def} 
becomes the stochastic equation problem (SE), that is to say, under \eqref{expected}, almost surely find $x^*\in\re^n$ such that $T(x^*)=0$.	
The SA methodology was first proposed by Robbins and Monro in \cite{rob.monro} for the SE problem in the case in which 
$T$ is the gradient of a strongly convex function under specific conditions. Since this fundamental work, 
SA approaches for SP and, more recently, for SVI, have been carried out 
\cite{Xu, nem, uday0, Uday, wang, philip, lan, uday2, uday3, uday4, uday5}. 
See also \cite{kushner, bach} for other problems where the stochastic approximation procedure is relevant 
(such as machine learning, online optimization, repeated games, queueing theory, signal processing and control theory).

\subsection{Related work on SA and contributions}\label{ss1.1}

The first SA method for SVI was analyzed in \cite{Xu}. Their method is:
\begin{eqnarray}\label{algorithm.xu1}
x^{k+1}=\Pi[x^k-\alpha_k F(\xi^k,x^k)],
\end{eqnarray} 
where $\Pi$ is the Euclidean projection onto $X$, $\{\xi^k\}$ is a sample of $\xi$ and $\{\alpha_k\}$ is 
a sequence of positive steps. In \cite{Xu}, the almost sure (a.s.) convergence 
is proved assuming $L$-Lipschitz continuity of $T$, strong monotonicity or strict monotonicity of $T$, stepsizes satisfying
$
\sum_k\alpha_k=\infty,\sum_k\alpha_k^2<\infty
$
(with $0<\alpha_k<2\rho/L^2$, assuming that  $T$ is $\rho$-strongly monotone), 
and an unbiased oracle with uniform variance, i.e., there exists $\sigma>0$ such that for all $x\in X$,
\begin{equation}\label{noise.bound2}
\esp\left[\Vert F(\xi,x)-T(x)\Vert^2\right]\le\sigma^2.
\end{equation}

After the above mentioned work, recent research on SA methods for SVI have been developed 
in \cite{nem, uday0, Uday, wang, philip, lan, uday2, uday3, uday4, uday5}. Two of the main concerns in these 
papers were the extension of the SA approach to the general monotone case and the obtention of (optimal) 
convergence rate and complexity results with respect to known metrics associated to the VI problem. In order 
to analyze the monotone case, SA methodologies based on the extragradient method of Korpelevich \cite{korpelevich} 
and the mirror-prox algorithm of Nemiroviski \cite{nem2} were used in \cite{nem, lan, uday2, uday3, uday4}, 
and iterative Tykhonov and proximal regularization procedures (see \cite{uday1, konnov}), 
were used in \cite{uday0, Uday, philip, uday5}. Other objectives in some of these papers were the use of 
incremental constraint projections in the case of difficulties accessing the feasible set \cite{wang, philip}, 
the convergence analysis in the absence of the Lipschitz constant \cite{uday0, uday2, uday5}, 
and the distributed solution of Cartesian variational inequalities \cite{uday0, Uday, philip, uday1}. 

In Cartesian variational inequalities, a network of $m$ agents is associated to a coupled variational inequality with constraint set
$
X=X^1\times\cdots\times X^m
$
and operator
$
F=(F_1,\ldots,F_m),
$
where the $i$-th agent is associated to a constraint set $X^i\subset\re^{n_i}$ and a map $F_i:\Xi\times\re^n\rightarrow\re^{n_i}$ such that $n=\sum_{i=1}^mn_i$. Two important problems which can be formulated as stochastic Cartesian variational inequalities are the stochastic \emph{Nash equilibria} and the stochastic \emph{multi-user optimization} problem. See \cite{Uday} for a precise definition. In these problems, the $i$-th agent has only access to constraint set $X^i$ and $F_i$ (which depends on other agents' decision sets) so that a distributed solution of the SVI is required. As an example, the distributed variant of method \eqref{algorithm.xu1} studied in \cite{uday0} takes the form: for all $i=1,\ldots,m$,
$$
x^{k+1}_i=\Pi_i\left[x^k_i-\alpha_{k,i}F_i(\xi^k_i,x^k)\right],
$$
where $\Pi_i$ is the Euclidean projection onto $X^i$. Thus, the $i$-th agent updates its decision evaluating his operator $F_i$ and projecting onto its decision set $X_i$.

%

In this paper we propose the following extragradient method: given $x^k$, define
\begin{eqnarray}
z^k&=&\Pi\Bigg[x^k-\frac{\alpha_{k}}{N_{k}}\sum_{j=1}^{N_{k}}F(\xi^k_{j},x^k)\Bigg],
\label{method:eq1}\\
x^{k+1}&=&\Pi\Bigg[x^k-\frac{\alpha_{k}}{N_{k}}\sum_{j=1}^{N_{k}}F(\eta^k_{j},z^k)\Bigg],
\label{method:eq2}
\end{eqnarray}
where $\{N_k\}\subset\mathbb{N}$ is a non-decreasing sequence and $\{\xi_j^k,\eta_j^k:k\in\mathbb{N}, j=1,\ldots, N_k\}$ 
are independent identically distributed (i.i.d.) samples of $\xi$. We call $\{N_k\}$ the \emph{sample rate} sequence. 

Next we make some observations regarding merit functions and complexity estimates. A merit function for 
VI$(T,X)$ is a non-negative function $f$ over $X$ such that $X^*=X\cap f^{-1}(0)$. An unrestricted 
merit function $f$ for VI$(T,X)$ is a merit function such that $X^*=f^{-1}(0)$. For any $\alpha>0$ 
we consider the \emph{natural residual function} $r_\alpha$, defined, for any $x\in\re^n$, by
$ 
r_\alpha(x):=\Vert x-\Pi(x-\alpha T(x))\Vert.
$ 
It is well known that $r_\alpha$ is an unrestricted merit function for VI$(T,X)$. Given $\epsilon>0$, 
we consider an iteration index $K=K_\epsilon$ (whose existence will be proved in Section \ref{ss3.4}), such that 
$
\mathbb{E}[r_\alpha(x^K)^2]<\epsilon,
$
and we look at 
$\mathbb{E}[r_\alpha(x^K)^2]$ as a {\it non-asymptotic convergence rate}. In particular, we
will have an $\mathcal{O}(1/K)$ convergence rate if 
$\mathbb{E}[r_\alpha(x^K)^2]\le Q/K$ for some constant $Q>0$ (depending on the initial iterate
and the parameters of the problem and the method). The (stochastic) {\it oracle complexity} will be defined as the total number of oracle
calls needed for $\mathbb{E}[r_\alpha(x^K)^2]<\epsilon$ to hold, i.e., $\sum_{k=1}^{K}2N_k$.  

Besides the natural residual, other merit functions were considered in prior work on SVI. 
Given a compact feasible set $X$, the \emph{dual gap-function} of VI$(T,X)$ is defined as $G(x):=\sup_{y\in X}\langle T(y),x-y\rangle$ for $x\in X$. 
In \cite{nem, lan, uday2, uday5}, rate of convergence were given in terms of the expected value of $G$ when $X$ is compact or, when 
$X$ is unbounded, in terms of the relaxed dual-gap function $\tilde G(x,v):=\sup_{y\in X}\langle T(y)-v,x-y\rangle$, 
introduced by Monteiro and Svaiter \cite{svaiter1, svaiter2}, based on the enlargement of monotone 
operators introduced in \cite{iusem3}. When $X$ is compact, the dual gap-function is a modification of the \emph{primal gap-function}, 
defined as $g(x):=\sup_{y\in X}\langle T(x),x-y\rangle$ for $x\in X$. Both the primal and dual gap-functions are 
continuous only if $X$ is compact. A variation suitable for unbounded feasible sets is 
the \emph{regularized gap-function}, defined, for fixed $a>0$, as
$
g_{a}(x):=\sup_{y\in X}\{\langle T(x),x-y\rangle-\frac{a}{2}\Vert x-y\Vert^2\},
$
for $x\in\re^n$. The regularized gap-function is continuous over $\re^n$. Another variation is 
the so called \emph{D-gap function}. It is defined, for fixed $b>a>0$, as
$
g_{a,b}(x):=g_a(x)-g_b(x),
$
for $x\in\re^n$. It is well known that $g_{a,b}:\re^n\rightarrow\re_+$ is a 
continuous unrestricted merit function for VI$(T,X)$. Moreover, the quadratic natural residual and 
the D-gap function are equivalent merit functions in the sense that, given $b>a>0$, there are constants $c_1,c_2>0$ such that for all $x\in\re^n$, $c_1 r_{b^{-1}}(x)^2\le g_{a,b}(x)\le c_2 r_{a^{-1}}(x)^2$ (see \cite{facchinei}, Theorems 10.2.3, 10.3.3 and Proposition 10.3.7). These properties hold independently of the compactness of $X$. 

Next we resume the contributions of the algorithm presented in this paper.

\medskip
\noindent
i) \textbf{Asymptotic-convergence}: Assuming  \emph{pseudo-monotonicity} of $F$, and using an 
extragradient scheme, without regularization, we prove that, almost surely, the generated sequence is bounded, 
its distance to the solution set converges to zero and its natural residual value converges to zero a.s. and in $L^2$. 
Note that monotonicity implies pseudo-monotonicity. See \cite{uday4} for examples where the more general setting of pseudo-monotonicity is relevant 
(stochastic fractional programming, stochastic optional pricing and stochastic economic equilibria). The sequence generated 
by our method also possesses a new stability feature: for $p=2$ or any $p\ge4$, if the random operator has finite $p$-moment 
then the sequence is bounded in $L^p$, and we are able to provide explicit upper bounds in terms of the problem parameters. 
Previous work required a bounded monotone operator, specific forms of (pseudo)-monotonicity 
(monotonicity with acute angle, pseudo-monotonicity-plus, strict pseudo-monotonicity,  
symmetric pseudo-monotonicity or strong pseudo-monotonicity as in \cite{uday3,uday4}), or regularization procedures. 
The disadvantage of regularization procedures in the absence of strong monotonicity is the need to introduce additional 
coordination between the stepsize sequence and the regularization parameters. Also, the regularization induces
a suboptimal performance in terms of rate and complexity (see \cite{uday5}).   

\medskip
\noindent
ii) \textbf{Faster convergence rate with oracle complexity efficiency}: To the best of our knowledge, our work is
 the first SA method for SVI with \emph{stepsizes bounded away from zero}. Such feature allows our method to achieve a faster 
convergence rate $\mathcal{O}(1/K)$ in terms of the mean-squared natural residual under plain pseudo-monotonicity
(with no regularization requirements). As a consequence, our method achieves a convergence rate of $\mathcal{O}(1/K)$ 
in terms of the mean D-gap function value of the generated sequence. In previous works, methods with diminishing 
stepsizes satisfying $\sum_k\alpha_k=\infty$, $\sum_k\alpha_k^2<\infty$ were used, achieving a $\mathcal{O}(1/K)$ rate in 
terms of the mean-squared distance to $X^*$, with more demanding monotonicity assumptions (namely, bounded strongly pseudo-monotone 
operators and bounded monotone weak-sharp VI) and a  rate $\mathcal{O}(1/\sqrt{K})$ in terms of mean gap function values of the 
ergodic average of the generated sequence in the case of bounded monotone operators. Importantly, our method preserves the 
optimal oracle complexity $\mathcal{O}(\epsilon^{-2})$ up to a first order logarithmic term. 
By accelerating the rate, we reduce the computational complexity (in terms of projection computations), preserving a near-optimal 
oracle complexity. It should be noted that such acceleration represents the closing of the gap from the stochastic to the deterministic 
and it is distinct in nature from the acceleration of differentiable convex optimization problems using 
Nesterov-type gradient methods. We provide explicit upper bounds for the rate and complexity in terms of the problem parameters. 
As a corollary of our result we provide new classes of SVIs for which a convergence rate of $\mathcal{O}(1/K)$ holds in terms of 
the mean-squared distance to the solution set (see Section \ref{s4}). We remark that for  compact $X$, it is possible to show that 
the proposed extragradient method achieves a rate $\mathcal{O}(\frac{\ln K}{K})$ in terms of the mean dual gap-function value of the 
ergodic average of the generated sequence with an optimal oracle complexity (up to a logarithmic factor). If different set of weights are used in the ergodic average (such as window-based averaging, Nesterov-like extrapolation and other schemes) then our method achieves a rate $\mathcal{O}(\frac{1}{K})$ with an optimal oracle complexity. See e.g. \cite{shapiro2, uday2, lan}. In the context of large dimension data ($n\gg1$), our algorithm complexity is independent of the dimension $n$ (see Proposition \ref{pp1}). 
	
\medskip
\noindent
iii) \textbf{Unbounded setting}: The results in items (i)-(ii) are valid for an \emph{unbounded feasible set} 
and \emph{unbounded operator}. Important examples of such a setting include complementarity problems and systems of equations. 
Asymptotic convergence for an unbounded feasible set is analyzed in \cite{uday0, wang, uday4, uday5} with more demanding 
monotonicity hypotheses, and in \cite{Uday, philip} for the monotone case, but with an additional regularization procedure. 
To the best of our knowledge, convergence rates in the case of an unbounded feasible set were treated only in \cite{wang, lan}. 
In \cite{wang}, a convergence rate is given only for strongly monotone operators. In \cite{lan}, assuming \emph{uniform variance} 
over $X$ (in the sense of \eqref{noise.bound2}), a convergence rate of $\mathcal{O}(1/\sqrt{K})$ for the ergodic average of the 
iterates is achieved in terms of the mean value of a relaxed gap function recently introduced by Monteiro and
 Svaiter \cite{svaiter1}-\cite{svaiter2}. It should be noted however that, even when assuming uniform variance, the sequence 
of the iterates generated by their method may diverge to $\infty$ (see Example \ref{example1}). Our convergence analysis in 
items (i)-(ii) does not depend upon boundedness assumptions, and we prove the faster rate of $\mathcal{O}(1/K)$ in terms of the 
mean (quadratic) natural residual and the mean D-gap function, which are new results. The natural residual and the D-gap 
function are better behaved than the (standard) gap function: 
the former are finite valued and Lipschitz continuous over $\re^n$, 
while the later is finite valued and continuous only for a compact $X$.  

\medskip
\noindent
iv) \textbf{Non-uniform variance}: Accordingly to what we know, all previous works require that the variance 
of the oracle error be \emph{uniform} over $X$ (in the sense of \eqref{noise.bound2}), excepting in \cite{wang} for the strongly 
monotone case, and in \cite{philip} for the case of a weak-sharp monotone operator, and also  for the monotone case with an 
iterative Tykhonov regularization (with no convergence rate results). Such uniform variance assumption holds for bounded operators, 
but not for unbounded ones, on an unbounded feasible set. Typical situations where this assumption fails to hold include affine 
complementarity problems and systems of equations. In such cases, the variance of the oracle error tends (quadratically) 
to $\infty$ in the horizon (see Example \ref{example2}). The performance of our method, in terms of the oracle complexity, 
depends on the point $x^*\in X^*$ with \emph{minimal} trade-off between variance and distance to initial iterates 
``ignoring'' points with high variance (see comments after Theorem \ref{thm:convergence:rate} and Section \ref{sss3.4.1}). 
This result also improves over the case in which \eqref{noise.bound2} \emph{does} holds but $\sigma(x^*)^2\ll\sigma^2$ or, over 
the case in which $X$ is compact but $\Vert x^0 -x^*\Vert\ll\diam(X)$. In conclusion, the performance of 
method \eqref{method:eq1}-\eqref{method:eq2} depends on solution points $x^*$ with minimal variance, compared to the conservative
 upper bound $\sigma^2$, and minimal distance to initial iterates. In the case of uniform variance over $X^*$ or $X$, 
we obtain sharper estimates of rate and complexity in item (ii). 		

\medskip
\noindent
v) \textbf{Distributed solution of multi-agent system}: The analysis in items (i)-(iv) also holds true for 
the distributed solution of stochastic Cartesian variational inequalities, in the spirit of \cite{uday0, Uday, philip, uday1}. 
In our framework (see Algorithm \eqref{algorithm.extra.cte1}-\eqref{algorithm.extra.cte2}), agents update synchronous 
stepsizes bounded away from zero over a range $(0,\mathcal{O}(1)L^{-1})$. An advantage of the extragradient approach in the 
distributed case is that we do not require iterative regularization procedures as in \cite{Uday, philip, uday1}, for coping with 
the merely monotone case. This implies that the faster convergent rate of $\mathcal{O}(1/K)$ is achievable with a near-optimal oracle
complexity under weaker conditions (such as unbounded set and non-uniform variance). As discussed later on, our 
algorithm requires the choice of a 
sampling rate for dealing with the setting of items (i)-(iv). Hence, in the distributed solution case, 
agents should have the choice of sharing their oracle calls or not, and we allow both options. 
In the later case of fully distributed sampling, the oracle complexity has higher order dependence in terms of the network 
dimension $m$, which may be demanding in the context of large networks ($m\gg1$). For this case, 
if an estimate of $m$ is available and a decreasing sequence of (deterministic) 
parameters $\{b_i\}_{i=1}^m$ is shared (in any order) among agents, then our algorithm has oracle complexity of 
order $m (a^{-1}\epsilon^{-1})^{2+a}$ for arbitrary $a>0$ (up to a scaling factor in the sample rate). See 
Proposition \ref{pp2}. Further dimension reduction possibilities will be the subject of future work.

For achieving the results of items (i)-(v), we employ an iterative \emph{variance reduction} procedure. This means that, 
instead of calling the oracle once per iteration (as in previous SA methods for SVI studied so far), 	
our method calls the oracle $N_k$ times at iteration $k$ and uses the associated empirical average of the values of the random operator 
$F$ at the current iterates $x^k$ and $z^k$ (see \eqref{method:eq1}-\eqref{method:eq2}). Since the presence of the stochastic error destroys
the strict Fej\'er property (satisfied by the generated sequence in the deterministic setting), the mentioned variance reduction 
procedure is the mechanism that allows our extragradient method to converge in an unbounded setting with stepsizes bounded away from zero, 
and to achieve an accelerated rate in terms of the natural residual. Such variance reduction scheme is efficient since we maintain a 
near-optimal oracle complexity when compared to the classical SA method. Precisely, given a prescribed tolerance $\epsilon>0$, the 
classical SA method requires $\mathcal{O}(\epsilon^{-2})$ iterations, $\mathcal{O}(\epsilon^{-2})$ samples and a final ergodic average
 of size $\mathcal{O}(\epsilon^{-2})$. As will be seen in Prop. \ref{pp1}, our method requires $K:=\mathcal{O}(\epsilon^{-1})$ 
iterations, $\mathcal{O}(\epsilon^{-2})$ samples and, for $k\le K$, the $k$-th iteration computes an empirical average of size 
$k$ (up to a first order logarithmic factor). Hence the total cost in averaging is also $\mathcal{O}(\epsilon^{-2})$ (again, 
up to a first order logarithmic 
factor). In conclusion, our method uses the same amount of total samples and same effort in averaging as in the classical SA method but 
with empirical averages with smaller sizes distributed along iterations instead of one ergodic average at the final iteration. This is 
the reason for improving the required number of iterations from $\mathcal{O}(\epsilon^{-2})$ to $\mathcal{O}(\epsilon^{-1})$ and 
thus reducing the number of projections by one order.\footnote{The possibility of distributing the empirical averages along iterations 
is possible due to the on-line nature of the SA method. This is not shared by the SAA methodology which is an off-line procedure.} 
The use of empirical averages along iterations is also the reason we can include unbounded operators, oracles with non-uniform variance 
and give estimates which depend on the variance at points of the trajectory of the method and at points of $X^*$ (but not on the whole $X$). 
Such results are not shared by the SAA method and SA with constant $N_k$. In order to obtain these results, we use martingale 
moment inequalities and a supermartingale convergence theorem (see Section \ref{ss2.2}). Our sampling procedure also possesses 
a \emph{robust} property: a scaling factor on the sampling rate maintains the progress of the algorithm with proportional scaling in 
the convergence rate and oracle complexity (see Propositions \ref{pp1}, \ref{pp3} and \ref{pp2}, and \cite{shapiro2} for robust methods). 
In Examples \ref{example2} and \ref{example1} of Section \ref{ss3.2} 
we show typical situations where such variance reduction procedure is relevant or even necessary. 

To the best of our knowledge the variance reduction procedure mentioned above is new for SA solution of SVI. Moreover, it seems 
that the derivation of the faster rate of $\mathcal{O}(1/K)$ with a near-optimal stochastic oracle complexity, an unbounded feasible 
set and an oracle with non-uniform variance is also new for convex stochastic programming. During the preparation of this paper we 
became aware of references \cite{wainwright, nocedal, friedlander, homem-de-mello, ferris, shanbhag:blanchet}, 
where variable sample-size methods are studied for stochastic optimization. We treat the general case of pseudo-monotone 
variational inequalities with weaker assumptions. Also, our analysis differs somewhat from these works relying on martingale 
and optimal stopping techniques. In \cite{wainwright, nocedal, friedlander, shanbhag:blanchet} the SA approach is studied
for convex stochastic optimization problems. In \cite{nocedal, shanbhag:blanchet}, the focus is on 
gradient descent methods applied to strongly convex optimization problems. In \cite{friedlander} the strong convexity property 
is slightly weakened by assuming a special error bound on the solution set (which is satisfied by strongly convex optimization 
problems in particular). In \cite{nocedal, friedlander} the optimization problem is unconstrained while in \cite{shanbhag:blanchet} 
the problem has a compact feasible set.
In \cite{nocedal},  second order information is assumed and an adaptive sample size selection is used. In \cite{friedlander} uniform 
boundedness assumptions are required. In \cite{wainwright}, a variant of the dual averaging method of Nesterov \cite{nesterov}
 is applied for solving non-smooth stochastic convex optimization, assuming a compact feasible set and uniform variance. A constant 
oracle call per iteration $N_k\equiv N>1$ is used,  obtaining a convergence rate of  $\mathcal{O}(1/\sqrt{KN})$ for the ergodic 
average of the sequence, while we typically use $N_k=\mathcal{O}(k(\ln k)^{1+b})$ with $b>0$ obtaining a rate of $\mathcal{O}(1/K)$ 
for the generated sequence. In \cite{homem-de-mello, ferris},  
the SAA approach for stochastic optimization is studied. This is an implicit method, unlike the SA methodology. Also, 
uniform boundedness assumptions are required. In \cite{ferris} the focus is on unconstrained optimization, with second order information,
using Bayesian analysis for an adaptive choice of
$N_k$. See also \cite{ghadimi:lan:zhang}.

The paper is organized as follows: in Section \ref{s2} we present notation and preliminaries, 
including the required probabilistic tools. In Section \ref{s3} we present the algorithm and its convergence analysis. 
In Subsection \ref{ss3.1} the algorithm is formally presented while in Subsection \ref{ss3.2} the assumptions required for its 
analysis are discussed. Subsection \ref{ss3.3} presents the convergence analysis while Subsection \ref{ss3.4} focuses on convergence rates and complexity results.

\section{Preliminaries}\label{s2}

\subsection{Projection operator and notation}

For $x,y\in\re^n$, we denote $\langle x,y\rangle$ the standard inner product, and $\Vert x\Vert=\sqrt{\langle x,x\rangle}$ 
the correspondent Euclidean norm. Given $C\subset\re^n$ and $x\in\re^n$, we use the 
notation $\dist(x,C):=\inf\{\Vert x-y\Vert:y\in C\}$. For a closed and 
convex set $C\subset\mathbb{R}^n$, we use the notation 
$\Pi_{C}(x):=\argmin_{y\in C}\Vert y-x\Vert^2$ for $x\in\re^n$. Given $H:\re^n\rightarrow\re^n$, S$(H,C)$ denotes the solution 
set of VI$(H,C)$. For a matrix $B\in\re^{n \times n}$, we use the notation $\Vert B\Vert:=\sup_{x\neq 0}\Vert Bx\Vert/\Vert x\Vert$. 
We use the notation $[m]:=\{1,\ldots,m\}$ for $m\in\mathbb{N}$ and $(\alpha_i)_{i=1}^m:=(\alpha_1,\ldots,\alpha_m)$ 
for $\alpha_i\in\re$ and $i\in[m]$. We also use the notation $\mathbb{N}_0:=\mathbb{N}\cup\{0\}$. 
We use the abbreviation ``RHS'' for ``right hand side''. Given sequences $\{x^k\}$ and $\{y^k\}$, we use the 
notation $x^k=\mathcal{O}(y^k)$ or $\Vert x^k\Vert\lesssim\Vert y^k\Vert$ to mean that there exists a constant $C>0$ 
such that $\Vert x^k\Vert\le C\Vert y^k\Vert$ for all $k$. The notation 
$\Vert x^k\Vert\sim\Vert y^k\Vert$ means that $\Vert x^k\Vert\lesssim\Vert y^k\Vert$ and $\Vert y^k\Vert\lesssim\Vert x^k\Vert$. 
Given a $\sigma$-algebra $\alg$ and a  
random variable $\xi$, we denote by $\esp[\xi]$, $\esp[\xi|\alg]$, and $\var[\xi]$, the expectation, conditional expectation and 
variance, respectively. We denote by $\covar[B]$ the covariance of a random vector $B$. Also, we write $\xi\in\alg$ 
for ``$\xi$ is $\alg$-measurable''. We denote by $\sigma(\xi_1,\ldots,\xi_k)$ the $\sigma$-algebra generated by the 
random variables $\xi_1,\ldots,\xi_k$. Given the random variable $\xi$ and $p\ge1$, $\Lpnorm{\xi}$ is the $L^p$-norm of $\xi$ and 
$
\Lpnorm{\xi\,|\alg}:=\sqrt[p]{\esp\left[|\xi|^p\,|\alg\right]}
$
is the $L_p$-norm of $\xi$ conditional to the $\sigma$-algebra $\alg$. $N(\mu,\sigma^2)$ denotes the normal distribution 
with mean $\mu$ and variance $\sigma^2$. Given  $x\in\re$, we denote by $x_+:=\max\{0,x\}$ its positive part and by $\lceil x\rceil$ 
the smallest integer greater or equal to $x$.

The following properties of the projection operator are well known; see Chapter 1 of \cite{facchinei}.
\begin{lema}\label{proj}
Take a non-empty closed and convex set $C\subset\mathbb{R}^n$. Then
\begin{itemize}
\item[i)] Given $x\in\mathbb{R}^n$, $\Pi_{C}(x)$ is the unique point of $C$ satisfying the following property: 
$\langle x-\Pi_{C}(x),y-\Pi_{C}(x)\rangle\le0$, for all $y\in C$.
\item[ii)] For all $x\in\mathbb{R}^n, y\in C$,
$
\Vert \Pi_{C}(x)-y\Vert^2+\Vert \Pi_{C}(x)-x\Vert^2\le\Vert x-y\Vert^2.
$	
\item[iii)]For all $x,y\in\mathbb{R}^n$,
$
\Vert \Pi_{C}(x)-\Pi_{C}(y)\Vert\le\Vert x-y\Vert.
$	
\item[iv)]Given $\alpha>0$ and $H:\re^n\rightarrow\re^n$, $\mbox{\emph{S}}(H,C)=\{x\in\re^n:x=\Pi_{C}[x-\alpha H(x)]\}$.
\end{itemize}
\end{lema}

\subsection{Probabilistic tools}\label{ss2.2}
As in other stochastic approximation methods, a fundamental tool to be used is the following convergence theorem of 
Robbins and Siegmund (see \cite{rob}).
\begin{thm}\label{rob}
Let $\{y_k\},\{u_k\}, \{a_k\}, \{b_k\}$ be sequences of nonnegative integrable 
random variables, adapted to the filtration $\{\alg_k\}$, such that a.s. for all $k\in\mathbb{N}$,
$
\esp\big[y_{k+1}\big| \alg_k\big]\le(1+a_k)y_k-u_k+b_k,
$
$\sum a_k<\infty$ and $\sum b_k<\infty$. Then, a.s. $\{y_k\}$ converges and $\sum u_k<\infty$.
\end{thm}

If a.s. for all $k\in\mathbb{N}$, $\esp[y_{k+1}|\alg_k]=y_k$ then $\{y_k,\alg_k\}$ 
is called a \emph{martingale}. We shall also need the following moment inequality; see \cite{burk,marinelli}.
\begin{thm}[The inequality of Burkholder-Davis-Gundy]\label{t1}
We denote as $\Vert\cdot\Vert$ the Euclidean norm in $\re^n$. For all $\qc\ge1$, there exists $C_\qc>0$ 
such that for any vector-valued martingale $(y_i,\alg_i)_{i=0}^N$ taking values in $\re^n$ with $y_0=0$, it holds that
\begin{equation}\label{ee4}
\Lqcnorm{\Vert y_N\Vert}\le\Lqcnorm{\sup_{i\leq N}\Vert y_i\Vert}\leq C_\qc\,\Lqcnorm{\sqrt{\sum_{k=1}^N\Vert y_i-y_{i-1}\Vert^2}}.
\end{equation}
\end{thm}
For $\qc\ge2$, we will use the rightmost inequality in \eqref{ee4} in the following simpler form, 
which follows from applying Minkowski's inequality ($\qc/2\ge1$):
\begin{equation}\label{BDG}
\Lqcnorm{\Vert y_N\Vert}\le\Lqcnorm{\sup_{i\leq N}\Vert y_i\Vert}\leq C_\qc\,\sqrt{\sum_{k=1}^N\,\Lqcnorm{\Vert y_i-y_{i-1}\Vert}^2}.
\end{equation} 

\section{An extragradient method with stepsizes bounded away from zero}\label{s3}

\subsection{Statement of the algorithm}\label{ss3.1}

Our extragradient method takes the form:

\quad

\begin{algo}[Stochastic extragradient method]
\quad

\begin{enumerate}
\item{\bf Initialization:} Choose the initial iterate $x^0\in\mathbb{R}^n$, 
a positive stepsize sequence $\{\alpha_k\}$, 
the sample rate $\{N_k\}$ and initial samples $\{\xi^0_{j}\}_{j=1}^{N_{0}}$ 
and $\{\eta^0_{j}\}_{j=1}^{N_{0}}$ of the random variable $\xi$.
\item{\bf Iterative step:} Given iterate $x^k$, generate 
samples $\{\xi^k_{j}\}_{j=1}^{N_{k}}$ and $\{\eta^k_{j}\}_{j=1}^{N_{k}}$ of $\xi$ and define:
\begin{eqnarray}
z^k&=&\Pi\Bigg[x^k-\frac{\alpha_{k}}{N_{k}}\sum_{j=1}^{N_{k}}F(\xi^k_{j},x^k)\Bigg],
\label{algorithm.extra.cte1.m1}\\
x^{k+1}&=&\Pi\Bigg[x^k-\frac{\alpha_{k}}{N_{k}}\sum_{j=1}^{N_{k}}F(\eta^k_{j},z^k)\Bigg].
\label{algorithm.extra.cte2.m1}
\end{eqnarray}
\end{enumerate}
\end{algo}
In  
\eqref{algorithm.extra.cte1.m1} and
\eqref{algorithm.extra.cte2.m1},
$\Pi$ is the Euclidean projection operator onto $X$. Method \eqref{algorithm.extra.cte1.m1}-\eqref{algorithm.extra.cte2.m1} 
is designed so that at iteration $k$ the random variable 
$\xi$ is sampled $2N_k$ times and the empirical average of $F$ at $x$ is used as the approximation of 
$T(x)$ at each projection step.

In order to incorporate the distributed case mentioned in Section \ref{ss1.1}, item(v), we will also analyze the case in 
which the SVI has a Cartesian structure. We consider the decomposition $\re^n=\prod_{i=1}^m\re^{n_i}$, 
with  $n=\sum_{i=1}^mn_i$, and furnish this 
space with the direct inner product $\langle x,y\rangle:=\sum_{i=1}^m\langle x_i,y_i\rangle$ for $x=(x_i)_{i=1}^m$ i
and $y=(y_i)_{i=1}^m$. We suppose that the feasible set has the form
$
X=\prod_{i=1}^mX^i,
$
where $X^i\subset\re^{n_i}$ is a closed and convex set for $i\in[m]$. The random operator $F:\Xi\times\re^n\rightarrow\re^n$ has the form 
$
F=(F_1,\ldots,F_m),
$
where $F_i:\Xi\times\re^n\rightarrow\re^{n_i}$ for $i\in[m]$. 
Given $i\in[m]$, we denote by $\Pi_i:\re^{n_i}\rightarrow\re^{n_i}$ the orthogonal projection onto $X^i$. 
We emphasize that the orthogonal projection under a Cartesian structure has a simple form: 
for $x=(x_i)_{i=1}^m\in\re^n$, we have
$
\Pi_X(x)=(\Pi_{X^1}(x_1),\ldots,\Pi_{X^m}(x_m)).
$

In such a setting, the method takes the form:

\quad

\begin{algo}[Stochastic extragradient method: distributed case]\label{algorithm:2}
\quad

\begin{enumerate}
\item{\bf Initialization:} Choose the initial iterate $x^0\in\mathbb{R}^n$, 
the stepsize sequence $\alpha_k>0$, 
the sample rates $N_k=(N_{k,i})_{i=1}^m\in\mathbb{N}^m$ and, for each $i\in[m]$, generate the initial samples $\{\xi^0_{j,i}\}_{j=1}^{N_{0,i}}$ 
and $\{\eta^0_{j,i}\}_{j=1}^{N_{0,i}}$ of the random variable $\xi$.
\item{\bf Iterative step:} Given $x^k=(x^k_i)_{i=1}^m$, for each $i\in[m]$, generate samples $\{\xi^k_{j,i}\}_{j=1}^{N_{k,i}}$ and $\{\eta^k_{j,i}\}_{j=1}^{N_{k,i}}$ of $\xi$ and define:
\begin{eqnarray}
z^k_i&=&\Pi_i\Bigg[x^k_i-\frac{\alpha_{k}}{N_{k,i}}\sum_{j=1}^{N_{k,i}}F_i(\xi^k_{j,i},x^k)\Bigg],\label{algorithm.extra.cte1}\\
x^{k+1}_i&=&\Pi_i\Bigg[x^k_i-\frac{\alpha_{k}}{N_{k,i}}\sum_{j=1}^{N_{k,i}}F_i(\eta^k_{j,i},z^k)\Bigg].\label{algorithm.extra.cte2}
\end{eqnarray}
\end{enumerate}
\end{algo}

Method \eqref{algorithm.extra.cte1.m1}-\eqref{algorithm.extra.cte2.m1} is a particular case of method 
\eqref{algorithm.extra.cte1}-\eqref{algorithm.extra.cte2} with $m=1$. The only additional requirement when $m>1$ 
is the sampling coordination between agents (Assumption \ref{extra.cte.unbias}). We  define next the stochastic errors: for each $i\in[m]$, 
\begin{eqnarray}
\epsilon^k_{1,i} &:=&\frac{1}{N_{k,i}}\sum_{j=1}^{N_{k,i}}F_i(\xi^k_{j,i},x^k)-T_i(x^k),\label{extra.cte.noise.sample1}\\
\epsilon^k_{2,i} &:=&\frac{1}{N_{k,i}}\sum_{j=1}^{N_{k,i}}F_i(\eta^k_{j,i},z^k)-T_i(z^k),\label{extra.cte.noise.sample2}
\end{eqnarray}
in which case method (\ref{algorithm.extra.cte1})-(\ref{algorithm.extra.cte2}) 
is expressible in a compact form as:
\begin{eqnarray}
z^k&=&\Pi[x^k-\alpha_k(T(x^k)+\epsilon^k_1)],\label{algorithm.extra.cte1.F}\\
x^{k+1}&=&\Pi[x^k-\alpha_k(T(z^k)+\epsilon^k_2)]\label{algorithm.extra.cte2.F},
\end{eqnarray}
where $\Pi:\re^n\rightarrow\re^n$ is the projection operator onto $X$ and
$\epsilon^k_l:=(\epsilon^k_{l,i})_{i=1}^m$ for $l\in\{1,2\}$.

\subsection{Discussion of the assumptions}\label{ss3.2}
For simplicity of notation, we aggregate the samples as
\begin{eqnarray*}
\xi^k_i:=\{\xi^k_{j,i}:j\in[N_{k,i}]\},\mbox{  }\xi^k:=\{\xi^k_i:i\in[m]\},\\
\eta^k_i:=\{\eta^k_{j,i}:j\in[N_{k,i}]\},\mbox{  } \eta^k:=\{\eta^k_i:i\in[m]\}.
\end{eqnarray*}
In the method \eqref{algorithm.extra.cte1.F}-\eqref{algorithm.extra.cte2.F}, the sample $\{\xi^k\}$ is used 
in the first projection while $\{\eta^k\}$ is used in the second projection. In the case of a Cartesian SVI, 
$\{\xi^k_i\}$ and $\{\eta^k_i\}$ are the samples used in the first and second projections in 
\eqref{algorithm.extra.cte1}-\eqref{algorithm.extra.cte2} by the $i$-th agent respectively. 

We shall study the stochastic process $\{x^k\}$ with respect to the filtrations
$$
\alg_k=\sigma(x^0,\xi^0,\ldots,\xi^{k-1},
\eta^0,\ldots,\eta^{k-1}),\quad
\widehat\alg_k=\sigma(x^0,\xi^0,\ldots,\xi^k,
\eta^0,\ldots,\eta^{k-1}).
$$
We observe that by induction, $x^k\in\alg_k$ and $z^k\in\widehat\alg_k$ but $z^k\notin\alg_k$.
The filtration $\alg_k$ corresponds to the information carried until iteration $k$, 
to be used on the computation of iteration $k+1$. The filtration $\widehat\alg_k$ corresponds to the information carried until iteration $k$ plus the information produced at the first projection step of iteration $k+1$, namely, $\widehat\alg_k=\sigma(\alg_k\cup\sigma(\xi^{k}))$. The way information evolves according to filtrations $\{\alg_k,\widehat\alg_k\}$ is natural in applications. Also, the use of two filtrations will be important since even though $z^k\notin\alg_k$ we have $z^k\in\widehat\alg_k$, so that, given $i\in[m]$:
\begin{eqnarray}
\esp[\epsilon^k_{2,i}|\widehat\alg_k]&=
&\esp\Bigg[\frac{1}{N_{k,i}}\sum_{j=1}^{N_{k,i}}F_i(\eta_{j,i}^k,z^k)-T_i(z^k)\Bigg|\widehat\alg_k\Bigg]\nonumber\\
&=&\frac{1}{N_{k,i}}\sum_{j=1}^{N_{k,i}}\esp\big[F_i(\eta_{j,i}^k,z^k)\big|\widehat\alg_k\big]-T_i(z^k)\nonumber\\
&=&\frac{1}{N_{k,i}}\sum_{j=1}^{N_{k,i}}T_i(z^k)-T_i(z^k)=0,\label{mean.error}
\end{eqnarray}
if for every $i\in[m]$, $\{\eta^k_{j,i}:j\in[N_{k,i}]\}$ is independent of $\widehat\alg_k$ and 
identically distributed as $\xi$. We exploit \eqref{mean.error} for avoiding first order moments of the stochastic errors, which drastically diminishes the complexity by an order of one, and for using martingale techniques.\footnote{If also $\{\xi^k_{j,i}:j\in[N_{k,i}]\}$ is independent of $\alg_k$ and identically distributed as $\xi$, then, for $i\in[m]$,
$
\var[\epsilon^k_{1,i}]=N_{k,i}^{-1}\var[F_i(\xi,x^k)]
$
and 
$
\var[\epsilon^k_{2,i}]=N_{k,i}^{-1}\var[F_i(\xi,z^k)],
$
so that our method iteratively reduces the variance of the oracle error as long as $\{N_{k,i}\}_{k\in\mathbb{N}}$ increases.
} 
We remark that, with some minor extra effort, the same samples can be used in both projections in method \eqref{algorithm.extra.cte1}-\eqref{algorithm.extra.cte2}.
Next we describe the assumptions required in our convergence analysis. 

\begin{assump}[Consistency]\label{existence}
The solution set $X^*:=\mbox{\emph{S}}(T,X)$ is non-empty.
\end{assump}

\begin{assump}[Stochastic model]\label{boundedness}
$X\subset\re^n$ is closed and convex, $(\Xi,\mathcal{G})$ is a measurable space such 
that $F:\Xi\times X\rightarrow\re^n$ is a Carath\'eodory map,\footnote{That is, $F(\xi,\cdot):X\rightarrow\re^n$ is continuous for a.e. $\xi\in\Xi$ and $F(\cdot,x):\Xi\rightarrow\re^n$ is measurable.
}
 $\xi:\Omega\rightarrow\Xi$ is a random variable 
defined on a probability space $(\Omega,\alg,\prob)$ and $\esp[\Vert F(\xi,x)\Vert]<\infty$ for all $x\in X$.
\end{assump}

\begin{assump}[Lipschitz continuity]\label{extra.cte.lipschitz}
The mean operator $T:X\rightarrow\re^n$ defined by \eqref{expected}
is Lipschitz continuous with modulus $L>0$. 
\end{assump}

\begin{assump}[Pseudo-monotonicity]\label{extra.cte.monotonicity}
The mean operator $T:\re^n\rightarrow\re^n$ is pseudo-monotone,\footnote{Pseudo-monotonicity is a weaker assumption than monotonicity, i.e., $\langle T(z)-T(x),z-x\rangle\ge 0$ for all $x,z\in\re^n$.} i.e., 
$
\langle T(x),z-x\rangle\ge 0\Longrightarrow\langle T(z),z-x\rangle\ge 0
$
for all $z,x\in\re^n$.
\end{assump}
 
\begin{assump}[Sampling rate]\label{extra.cte.samples}
Given $\{N_k\}$ as in Algorithm \ref{algorithm:2}, define $N_{k,\min}:=\min_{i\in[m]}N_{k,i}$ and
$
\frac{1}{\mathcal{N}_k}:=\sum_{i=1}^m\frac{1}{N_{k,i}}
$. Then one of the conditions is satisfied:
\begin{itemize}
\item[i)] $\sum_{k=0}^\infty\frac{1}{\mathcal{N}_k}<\infty$,
\item[ii)] $\sum_{k=0}^\infty\frac{1}{N_{k,\min}}<\infty$.
\end{itemize}
\end{assump}

Typically a sufficient choice is, for $i\in[m]$:
$$
N_{k,i}=\Theta_{i}\cdot \left(k+\mu_i\right)^{1+a_i}\cdot\left(\ln \big(k+\mu_i\big)\right)^{1+b_i},
$$
for any $\Theta_i>0$, $\mu_i>0$ with $a_i>0$, $b_i\ge-1$ or $a_i=0$, $b_i>0$ (the latter is the minimum requirement). 
It is essential to specify choices of the above parameters that induce a practical complexity of 
method \eqref{algorithm.extra.cte1}-\eqref{algorithm.extra.cte2}, i.e., practical upper bounds on the total oracle complexity
$
\sum_{k=1}^K\sum_{i=1}^m2N_{k,i},
$
where $K$ is an estimate of the total number of iterations needed for achieving a given specified tolerance $\epsilon>0$. 
A convergence rate in terms of $K$ is also desirable. As commented after Theorem \ref{thm:convergence:rate}, 
our algorithm achieves an optimal accelerated rate $\mathcal{O}(1/K)$ and an optimal complexity $\mathcal{O}(\epsilon^{-2})$ 
up to a first order logarithmic term $\ln(\epsilon^{-1})$.\footnote{In large scale problems such as in machine learning, the 
dependence of the rate and complexity estimates on the dimension is relevant in the case of large constraint dimension ($n_i\gg1$) or 
large networks ($m\gg1$). We show that our method has complexity $\mathcal{O}(\sigma^2)$ which is independent of dimension, 
where $\sigma^2$ is the variance, even in the case of an unbounded feasible set and a non-uniform variance. Sharper constants are 
available in the case of uniform variance (see Proposition \ref{pp3}).
}

We offer two options of sampling coordination among the agents:

\begin{assump}[Sampling coordination]\label{extra.cte.unbias}
For each $i\in[m]$ and $k\in\mathbb{N}_0$, $\{\xi^k_i\}$ and $\{\eta^k_i\}$ are i.i.d. samples of $\xi$ 
such that $\{\xi^k_i\}$ and $\{\eta^k_i\}$ are independent of each other. Also, one of the two next coordination conditions is satisfied:
\begin{itemize}
\item[i)] (Centralized sampling) For all $i\in[m]$, $N_{k,i}\equiv N_k$, $\xi^k_i\equiv\xi^k$ and $\eta^k_i\equiv\eta^k$.

\item[ii)] (Distributed sampling) $\{\xi^k,\eta^k:k\in\mathbb{N}\}$ is an i.i.d. sample of $\xi$.
\end{itemize}
\end{assump}
We remark that, with some extra effort, it is possible to use the same samples in each projection step of 
the method \eqref{algorithm.extra.cte1}-\eqref{algorithm.extra.cte2}, that is, $\xi^k_i\equiv\eta^k_i$ 
for $k\in\mathbb{N}_0$ and $i\in[m]$. We ask for independence in Assumption \ref{extra.cte.unbias} in order to 
simplify the analysis. Both conditions (i) and (ii) in Assumption \ref{extra.cte.unbias} are the same for $m=1$.\footnote{When  
$m>1$, item (i) corresponds to the case where one stochastic 
oracle is centralized. In this case, less samples are required but the sampling process needs total coordination. 
Item (ii) corresponds to the other extreme case, where the agents have completely distributed oracles so that the 
sampling process of each agent is conducted independently. We do not explore the intermediate possibilities between (i) and (ii).
}
Assumption \ref{extra.cte.unbias} implies in particular that $\{\xi^k\}$ is independent of $\alg_k$, 
$\{\eta^k\}$ is independent of $\widehat\alg_k$ and both are identically distributed as $\xi$. 
In particular, for any $x\in\re^n$, $k\in\mathbb{N}$, $i\in[m]$, $j\in[N_{k,i}]$:
$$
\esp\left[F_i(\xi^k_{j,i},x)\Big|\alg_k\right]=\esp\left[F_i(\eta^k_{j,i},x)\Big|\widehat\alg_k\right]=T_i(x).
$$

\begin{assump}[Stepsize bounded away from zero]\label{extra.cte.step} The stepsize sequence $\{\alpha_k\}$ 
in Algorithm \ref{algorithm:2} satisfies
$$
0<\inf_{k\in\mathbb{N}}\alpha_{k}\le
\hat\alpha:=\sup_{k\in\mathbb{N}}\alpha_{k}<\frac{1}{\sqrt{6}L}. 
$$
\end{assump}

The following two sets of assumptions ensure that the variance of the error $F(\xi,x)-T(x)$ is controlled, 
so that (together with Assumption \ref{extra.cte.samples} on the sampling rate) boundedness is guaranteed, 
even in the case of an unbounded operator.

\begin{assump}[Variance control]\label{extra.cte.control}
There exists $p\ge2$, such that one of the following three conditions holds:
\begin{itemize}
\item[i)] There exist $x^*\in X^*$ and $\sigma(x^*)>0$ such that for all $x\in X$,
$$
\Lpnorm{\Vert F(\xi,x)-T(x)\Vert}\leq \sigma(x^*)\,(1+\|x-x^*\|).
$$

\item[ii)] There exists a locally bounded and measurable function $\sigma:X^*\rightarrow\re_+$ such that for all $x^*\in X^*$, $x\in X$, the inequality in (i) is satisfied.

\item[iii)] There exist positive sequence $\{\sigma_{l,i}:i\in[m],l\in[n_i]\}$ such that for all $i\in[m]$, $l\in[n_i]$, $x\in X$, 
$
\Lpnorm{F_{\ell,i}(\xi,x)-T_{\ell,i}(x)}\leq \sigma_{\ell,i},
$
where $F_{\ell,i}$ and $T_{\ell,i}$ are the components of $F_i$ and $T_i$ respectively.
\end{itemize}
\end{assump}

In item (iii) we define $\sigma^2:=\sum_{i=1}^m\sum_{\ell=1}^{n_i}\sigma_{\ell,i}^2.$ Note that when $p=2$, $\sigma(x^*)^2\,(1+\|x-x^*\|)^2$ in the case of (i)-(ii), and $\sigma^2$ in the case of item (iii), are, respectively, upper bounds on the variance of the components of $F(\xi,x)$. 
Items (i) and (ii) are essentially the same, excepting that (i) only requires the condition to hold at just 
one point $x^*\in X^*$ rather than on the entire solution set. Item (i) is sufficient for the analysis, 
but (ii) allows for sharper estimates in the case of unbounded feasible set and operator. 
Item (iii) allows for even sharper ones. In the sequel we shall denote $q:=p/2$.

For the important case in which the random operator $F$ is Lipschitz, 
both items (i)-(ii) are satisfied with a continuous $\sigma:X^*\rightarrow\re_+$. Namely, if for any $x,y\in\re^n$,
\begin{equation}\label{equation:random:lipschitz}
\Vert F(\xi,x)-F(\xi,y)\Vert\le \mathsf{L}(\xi)\Vert x-y\Vert,
\end{equation}
for some measurable $\mathsf{L}:\Xi\rightarrow\re_+$ with finite $L_p$-norm for some $p\ge2$, 
then Assumptions \ref{boundedness}-\ref{extra.cte.lipschitz} and \ref{extra.cte.control} hold with $L:=\esp[\mathsf{L}(\xi)]$ and
\begin{equation}\label{equation:variance}
\sigma(x^*):=\max\{\Lpnorm{\Vert F(\xi,x^*)-T(x^*)\Vert},\Lpnorm{\mathsf{L}(\xi)}+L\},
\end{equation}
for $x^*\in X^*$. Indeed, Assumption \ref{extra.cte.lipschitz} with $L:=\esp[\mathsf{L}(\xi)]$ 
follows from Jensen's inequality and \eqref{equation:random:lipschitz}. For establishing \eqref{equation:variance}, 
note that by Minkowski's inequality
\begin{eqnarray*}
\Lpnorm{\Vert F(\xi,x)-T(x)\Vert}&\le &\Lpnorm{\Vert F(\xi,x)-F(\xi,x^*)\Vert}+\Lpnorm{\Vert F(\xi,x^*)-T(x^*)\Vert}+\Vert T(x)-T(x^*)\Vert\\
&\le &
(\Lpnorm{\textsf{L}(\xi)}+L)\Vert x-x^*\Vert+\Lpnorm{\Vert F(\xi,x^*)-T(x^*)\Vert},\\
&\le &\sigma(x^*)(\Vert x-x^*\Vert+1),
\end{eqnarray*}
using \eqref{equation:random:lipschitz} and the fact that $T$ is $L$-Lipschitz continuous in the second inequality, 
and \eqref{equation:variance} in the third inequality. Thus, Assumption \ref{extra.cte.control}(i)-(ii) is 
merely a \emph{finite} variance assumption even for the case of an unbounded feasible set. 
Assumption \ref{extra.cte.control}(iii) means that the variance is \emph{uniformly bounded} over 
the feasible set $X$. It has been assumed in most of the past literature \cite{nem, Uday, uday0, lan, uday2, uday3, uday4, uday5} on 
stochastic approximation algorithms for SVI and stochastic programming.\footnote{Assumption \ref{extra.cte.control}(iii) has been weakened 
in previous works only in situations in which the operator satisfies more demanding
monotonicity conditions (strongly monotone operator in \cite{wang} and weak-sharp monotone operator in \cite{philip}) or when
the operator is merely monotone, but with additional Tykhonov regularization (as in \cite{philip}, without 
convergence rate results).
} 
Assumptions \ref{extra.cte.control}(i)-(ii) are much weaker than Assumption 
\ref{extra.cte.control}(iii) and, to the best of our knowledge, seem to be new for monotone operators or convex functions without regularization.

The next examples provide instances where Assumption \ref{extra.cte.control}(i)-(ii) and the iterative variance reduction in method \eqref{algorithm.extra.cte1}-\eqref{algorithm.extra.cte2} are relevant or even \emph{necessary} for the asymptotic convergence of the generated sequence, in the case of an unbounded feasible set (e.g., stochastic equations and stochastic complementarity problems).

\quad

\begin{example}[Linear SVI with unbounded feasible set]\label{example2}
\emph{
The following example is a typical situation of a non-uniform variance over a unbounded feasible set. 
It includes the cases of stochastic linear equations and complementarity problem. Let the random operator be: 
$$
F(\xi,x)=A(\xi)x,
$$
for all $x\in\re^n$,
where $A(\xi)$ is a random matrix whose entries have finite mean and variance, 
such that $\bar A:=\esp[A(\xi)]$ is nonnull and positive semidefinite. In this case, 
$T(x)=\bar A x$ ($x\in\re^n$) is monotone and linear. For all $x\in\re^n$,
$
\var[F(\xi,x)]=x^t B x,
$
where $B:=\sum_{i=1}^m\covar\left[A_i(\xi)\right]$ is positive semidefinite 
and $A_1(\xi),\ldots,A_m(\xi)$ are the rows of $A(\xi)$. We denote by $N(B)$ the kernel of $B$ and by $N(B)^{\perp}$ its orthogonal complement. Given $x\in\re^n$, let $x_B$ be the orthogonal projection of $x$ onto $N(B)^\perp$.
Then for all $x\in\re^n$ we have
$$
\var[F(\xi,x)]\ge\lambda_+(B)\Vert x_B\Vert^2,
$$
where $\lambda_+(B)$ is the smallest nonnull eigenvalue of $B$. In particular, if $B$ is positive definite, then for all $x\in\re^n$, 
$
\var[F(\xi,x)]\ge\lambda_{\min}(B)\Vert x\Vert^2,
$
where $\lambda_{\min}(B)$ is the smallest eigenvalue of $B$.
This shows that Assumption \ref{extra.cte.control}(iii) does not hold if $X$ is unbounded (in fact, the variance grows quadratically in the infinite horizon).
}
\end{example}

\quad

\begin{example}[Equation problem for zero mean random constant operator]\label{example1}
\emph{The following example presents a simple situation where, in the case of an unbounded feasible set and an oracle with \emph{uniform variance}, the method in \cite{lan} may possess an undesirable property: for the null operator $T\equiv0$, the method generates a sequence whose final ergodic average converges but the sequence itself a.s. diverges to $\infty$.} 

\emph{
The method in \cite{lan} say that given a prescribed number of iterations $K$, for $k\in[K]$ compute
\begin{eqnarray*}
z^k&=&\Pi\left[x^k-\alpha_k^K F(\xi^K_k,x^k)\right],\\
x^{k+1}&=&\Pi\left[x^k-\alpha_k^K F(\eta^K_k,z^k)\right],
\end{eqnarray*}
and give as final output the ergodic average
$
\bar z^K=\sum_{k=1}^K p_k^Kz^k, 
$
where $\{p^K_k\}$ is a positive sequence of weights such that $\sum_{k=1}p^K_k=1$ and $\{\alpha^K_k\}$ is a sequence of positive stepsizes. For an unbounded $X$, assuming uniformly bounded variance 
(Assumption \ref{extra.cte.control}(iii)) and a single oracle call per iteration, it is shown that there exists $\{v^K\}\subset\re^n$ such that $\esp[\tilde G(\bar z^K,v^K)]\lesssim1/\sqrt{K}$ and $\esp[\Vert v^K\Vert]\lesssim\sqrt{K}$ (see \cite{lan}, Corollary 3.4). In these statements, for $z,v\in\re^n$,
$
\tilde G(z,v)=\sup_{y\in X}\langle T(y)-v,z-y\rangle,
$
as mentioned in Subsection \ref{ss1.1}. The following example shows that  $\limsup_{K\rightarrow\infty} \Vert z^K\Vert=\infty$ with total probability.
}
\emph{We shall consider $n=1$, but one can easily generalize the argument for any $n>1$. Consider $X=\re$ and the random operator given by
$$
F(\xi,x)=\xi,
$$
for all $x\in\re$, where $\xi$ is a random variable with zero mean, finite variance $\sigma^2$ and finite  third moment 
(one could generalize the argument assuming finite $q$-moment for any $q>2$). 
In this case, trivially $T\equiv0$, $X^*=\re$ and Assumption \ref{extra.cte.control}(iii) holds. It is easy to check that the mirror-prox method in \cite{lan} gives, 
after $K$ iterations, for $k\in[K]$: 
\begin{equation}\label{example:z}
z^k=x^1-\sum_{i=1}^{k}\alpha^K_i\xi^K_i,
\quad
\bar z^K=\sum_{k=1}^Kp^K_kz^k,
\end{equation}
where $p^K_k=c_0\Gamma_K\alpha^K_k$, $\gamma_k(\Gamma_k\alpha_k^K)^{-1}\equiv c_0$ is a constant, 
$\gamma_k:=2(1+k)^{-1}$, $\{\Gamma_k\}$ is defined recursively as $\Gamma_1:=1$, $\Gamma_k:=(1-\gamma_k)\Gamma_{k-1}$ 
and the stepsize is 
$$
\alpha^K_k:=\frac{k}{3LK+\sigma K\sqrt{K-1}},
$$
(see \cite{lan}, Corollary 3.4). Using the expression of $\{p^K_k\}$ and $\sum_{k=1}^Kp^K_k=1$, we get
\begin{equation}\label{example:bar:z}	
\bar z^K=x^1-\sum_{k=1}^K\theta^K_k\cdot\xi^K_k,
\end{equation}
where $\theta^K_k:=c_0\Gamma_K\alpha^K_k\sum_{i=k}^K\alpha^K_i$. Note that $\Gamma_k=\frac{2}{k(k+1)}$ and
$$
\theta^K_k=\frac{c_0\Gamma_K k}{\left(3LK+\sigma K\sqrt{K-1}\right)^2}\sum_{i=k}^Ki
=\frac{c_0 k(K-k+1)(K+k)}{K(K+1)\left(3LK+\sigma K\sqrt{K-1}\right)^2}.
$$	
We have the following estimates:
$$
\bar s_K^2:=\sum_{k=1}^K\left(\theta^K_k\right)^2\sim K^{-3},\quad
s_K^2:=\sum_{k=1}^K\left(\alpha^K_k\right)^2\sim 1,\quad
\sum_{k=1}^K\left(\alpha^K_k\right)^3\sim K^{-\frac{1}{2}}\quad(\mbox{as }K\rightarrow\infty).
$$
}

\emph{
We will now invoke Lyapounov's criteria (\cite{bill}, Theorem 7.3) with $\delta=1$ for the sum $\sum_{k=1}^K\alpha^K_k\cdot\xi^K_k$ 
of independent random variables, obtaining 
$$
\lim_{K\rightarrow\infty}\frac{\esp\left[\left|\xi\right|^{3}\right]}{s_K^3}\sum_{k=1}^K\left(\alpha^K_k\right)^{3}=
\lim_{K\rightarrow\infty}\esp\left[\left|\xi\right|^{3}\right]K^{-\frac{1}{2}}=0.
$$
Hence $\left(\sigma s_K\right)^{-1}\sum_{k=1}^K\alpha^K_k\xi^K_k$ converges in distribution to $N(0,1)$. 
Therefore, there exists some constant $C>0$ such that for any $R>0$, 
\begin{eqnarray}
\prob\left(\limsup_{K\rightarrow\infty}z^K\ge R\right)&=&
\prob\left(\limsup_{K\rightarrow\infty}\sum_{k=1}^K\frac{\alpha^K_k}{\sigma s_K}\cdot\xi^K_k\ge CR\right)\nonumber\\
&\ge & \limsup_{K\rightarrow\infty}\prob\left(\sum_{k=1}^K\frac{\alpha^K_k}{\sigma s_K}\cdot\xi^K_k\ge CR\right)>0,\label{e40}
\end{eqnarray}
using \eqref{example:z} and $s_K\sim 1$ in the equality and Portmanteau's Theorem (\cite{durrett}, Theorem 3.2.5) 
in the inequality. For every $R>0$, the event $A_R:=[\limsup_{K\rightarrow\infty}z^K\ge R]$ is a tail event with positive probability and, hence, has total probability, invoking Kolmogorov's zero-one law (\cite{durrett}, Theorem 2.5.1). We conclude from \eqref{e40} that
$$
\prob\left(\limsup_{K\rightarrow\infty} z^K=\infty\right)=\lim_{R\rightarrow\infty}\prob\left(A_R\right)=1.
$$
This shows that $\{z^K\}$ diverges with total probability. From the 1-Series Theorem (\cite{durrett}, Theorem 2.5.3), $\sum_{K=1}^\infty\bar s_K^2<\infty$ and \eqref{example:bar:z}, we have that a.s. $\{\bar z^K\}$ converges. 
}
\end{example}

\subsection{Convergence Analysis}\label{ss3.3}

For any $x=(x_i)_{i=1}^m\in\re^{n}$ and $\alpha>0$, we denote the (quadratic) natural residual function by
$$
r_\alpha(x)^2:=\left\Vert x-\Pi\left[x-\alpha T(x)\right]\right\Vert^2
=\sum_{i=1}^m\left\Vert x_i-\Pi_i\left[x_i-\alpha T_i(x)\right]\right\Vert^2.
$$
We start with two key lemmas whose proofs are given in the Appendix. Define 
\begin{equation}\label{def:rho}
\rho_k:=1-6L^2\alpha_{k}^2,
\end{equation}
for any $k\in\mathbb{N}_0$. We define recursively, for $k\in\mathbb{N}_0$, $A_0:=0$,
\begin{equation}\label{def:A}
A_{k+1}:= A_{k}+(8+\rho_k)\alpha_{k}^2\Vert\epsilon^k_1\Vert^2+
8\alpha_{k}^2\Vert\epsilon^k_2\Vert^2,
\end{equation}
and, for $x^*\in X^*$, $M_0(x^*):=0$, 
\begin{equation}\label{def:M}
M_{k+1}(x^*):= M_{k}(x^*)+2\langle x^*-z^k,\alpha_k\epsilon^k_2\rangle.
\end{equation}

\quad

\begin{lema}[Recursive relation]\label{extra.cte.korpelevich}
Suppose that Assumption \ref{existence}, \ref{extra.cte.lipschitz} and \ref{extra.cte.step} hold. Then, almost surely, for all  $k\in\mathbb{N}$ and $x^*\in X^*$,
\begin{eqnarray*}
\Vert x^{k+1}-x^*\Vert^2\le \Vert x^k-x^*\Vert^2-\frac{\rho_k}{2}r_{\alpha_k}(x^k)^2+M_{k+1}(x^*)-M_k(x^*)+A_{k+1}-A_k.
\end{eqnarray*}
\end{lema}	

\begin{lema}[Error decay]\label{lema:indsums}
Consider Assumptions \ref{boundedness}-\ref{extra.cte.control}. 
For each $i\in[m]$ and $N_i\in\mathbb{N}$, let $\xi_i:=\left\{\xi_{j,i}:j\in[N_i]\right\}$ be an i.i.d. 
sample of $\xi$ with $1/\mathcal{N}:=\sum_{i=1}^m1/N_i$ and $N_{\min}:=\min_{i\in[m]}N_i$. For any $x\in X$ set
$$
\epsilon_i(x):=\sum_{j=1}^{N_i}\frac{F_i(\xi_{j,i},x)-T_i(x)}{N_i},\quad
\epsilon(x):=(\epsilon_1(x),\ldots,\epsilon_m(x)).
$$
If Assumption \ref{extra.cte.control}(i) hold for some $x^*\in X^*$,  then 
for all $x\in X, v\in\re^{n}$,
$$
\Lpnorm{\left\Vert\epsilon(x)\right\Vert}\leq \sqrt{\frac{\mathsf{A}}{\mathcal{N}}}C_p f(x,x^*),\quad
\Lpnorm{\langle v,\epsilon(x)\rangle}\leq 
\Vert v\Vert\sqrt{\frac{\mathsf{B}}{\mathcal{N}}}C_p f(x,x^*),
$$
where $f(x,x^*):=\sigma(x^*)(1+\|x-x^*\|)$ and
\begin{enumerate}
\item $\mathsf{A}=1$ if $m=1$ and $\mathsf{A}=2$ if $m>1$,

\item $\mathsf{B}=2$ if $m>1$ and $\{\xi_{j,i}:1\le i\le m,1\le j\le N_i\}$ is i.i.d.,

\item $\mathsf{B}=1$ if $m=1$ or if $m>1$ with $N_i\equiv N$, $\xi_{j,i}\equiv\xi_j$ for all $i\in[m]$.
\end{enumerate}

Moreover, if Assumption \ref{extra.cte.control}(iii) holds, then for all $x\in X, v\in\re^{n}$,
$$
\Lpnorm{\left\Vert\epsilon(x)\right\Vert}\leq \frac{C_p\sigma}{\sqrt{N_{\min}}},\quad
\Lpnorm{\langle v,\epsilon(x)\rangle}\leq \Vert v\Vert\frac{C_p\sigma}{\sqrt{N_{\min}}}.
$$
\end{lema}
In the remainder of the paper, we take $\mathsf{A}$ and $\mathsf{B}$ as given in Lemma \ref{lema:indsums} 
with $\{\mathcal{N}_k\}$, $\{N_{k,\min}\}$, $\sigma(\cdot)$ and $\sigma$ as 
given in Assumptions \ref{extra.cte.samples}, \ref{extra.cte.unbias} and \ref{extra.cte.control}.

\quad

The following two results will establish upper bounds on $A_{k+1}-A_{k}$ and $M_{k+1}(x^*)-M_k(x^*)$ in 
terms of $\Vert x^k-x^*\Vert^2$ for any $x^*\in X^*$. Under the Assumptions \ref{extra.cte.control}(i)-(ii) of 
non-uniform variance, we need first a bound of $\|x^*-z^k\|^2$ in terms of $\|x^k-x^*\|^2$.  We define:
\begin{eqnarray}
\mathsf{G}_{k,p}(x^*)&:=&\alpha_{k}C_p\sigma(x^*),\label{prop:zk:gk}\\
\mathsf{H}_{k,p}(x^*)&:=&\mathsf{G}_{k,p}(x^*)\sqrt{\frac{\mathsf{A}}{\mathcal{N}_k}},\label{prop:zk:hk}
\end{eqnarray}

\begin{prop}\label{prop:zk}
Consider Assumptions \ref{existence}-\ref{extra.cte.control}. If Assumption \ref{extra.cte.control}(i) hold for some $x^*\in X^*$, then
$$
\Lpnorm{\|z^k-x^*\|\big|\alg_k}\leq 
\left[1+L\alpha_{k}+\mathsf{H}_{k,p}(x^*)\right]\|x^k-x^*\| + \mathsf{H}_{k,p}(x^*).
$$

Moreover, if Assumption \ref{extra.cte.control}(iii) holds, then
$$
\Lpnorm{\Vert z^k-x^*\Vert\big|\alg_k}\leq (1+\mathcal{L}\alpha_{k})\,\|x^k-x^*\|+
\alpha_{k}\left(\mathcal{M}+\frac{C_p\sigma}{\sqrt{N_{k,\min}}}\right),
$$
with $\mathcal{L}=L$ and $\mathcal{M}=0$ or, alternatively, 
if $\sup_{x\in X}\Vert T(x)\Vert\le M<\infty$, with $\mathcal{L}=0$ and $\mathcal{M}=2M$.
\end{prop}
\begin{proof}
Recall that $z^k =\Pi[x^k-\alpha_k(T(x^k)+\epsilon^k_1)]$. By Lemma \ref{proj}(iv), we have $x^*=\Pi[x^*-\alpha_k T(x^*)]$. 

Consider first Assumption \ref{extra.cte.control}(i). By Lemma \ref{proj}(iii),
\begin{eqnarray}
\|x^*-z^k\|&\leq &\,\| x^*-x^k - \alpha_k(T(x^*)-T(x^k))+\alpha_k\epsilon^k_1\|\nonumber\\
&\leq & \,\|x^*-x^k\| +\alpha_k\| T(x^k)-T(x^*)\| + \alpha_k\,\left\|\epsilon^k_1\right\|\nonumber\\
&\leq & (1+L\alpha_{k})\,\|x^*-x^k\|+\alpha_{k}\,\left\|\epsilon^k_1\right\|,\label{prop:zk.eq1}
\end{eqnarray}
using the Lipschitz continuity of $T$ in the last inequality. Using $x^k\in\alg_k$ and taking $\Lpnorm{\cdot|\alg_k}$ in \eqref{prop:zk.eq1} we get from Minkowski's inequality, 
\begin{equation}\label{prop:zk.eq1.1}
\Lpnorm{\Vert z^k-x^*\Vert\big|\alg_k}\leq (1+L\alpha_{k})\,\|x^*-x^k\|+\alpha_{k}\Lpnorm{\left\|\epsilon^k_{1}\right\|\Big|\alg_k}.
\end{equation}
We now recall the definition of $\epsilon^k_1$ in \eqref{extra.cte.noise.sample1}. We have 
\begin{equation}\label{prop:zk.eq2}
\Lpnorm{\left\|\epsilon^k_{1}\right\|\Big|\alg_k}\leq
\left(\frac{\mathsf{A}}{\mathcal{N}_k}\right)^{\frac{1}{2}}C_p\sigma(x^*)\,(1+\|x^k-x^*\|),
\end{equation}
using Lemma \ref{lema:indsums}, $x^k\in\alg_k$ and the independence of $\xi^k$ with $\alg_k$. Relations \eqref{prop:zk.eq1.1}-\eqref{prop:zk.eq2} prove the required claim.

We now consider Assumption \ref{extra.cte.control}(iii). In this case, \eqref{prop:zk.eq1} may be replaced by
\begin{equation}\label{prop:zk.eq3}
\Vert x^*-z^k\Vert\leq (1+\mathcal{L}\alpha_{k})\,\|x^*-x^k\|+
\alpha_{k}\left(\mathcal{M}+\left\|\epsilon^k_{1}\right\|\right),
\end{equation}
with $\mathcal{L}$ and $\mathcal{M}$ as stated in the proposition.  Using $x^k\in\alg_k$ and taking $\Lpnorm{\cdot|\alg_k}$ in \eqref{prop:zk.eq3} we get from Minkowski's inequality, 
\begin{equation}\label{prop:zk.eq3.1}
\Lpnorm{\Vert z^k-x^*\Vert\big|\alg_k}\leq  (1+\mathcal{L}\alpha_{k})\,\|x^*-x^k\|+\alpha_{k}\left(\mathcal{M}+\Lpnorm{\left\|\epsilon^k_{1}\right\|\Big|\alg_k}\right).
\end{equation}
By Lemma \ref{lema:indsums}, relation \eqref{prop:zk.eq2} is replaced by 
$
\Lpnorm{\left\|\epsilon^k_{1}\right\|\Big|\alg_k}\leq
\frac{C_p\sigma}{\sqrt{N_{k,\min}}}.
$
The claim results from the previous inequality and \eqref{prop:zk.eq3.1}.
\end{proof}

\quad

The following proposition gives bounds on the increments of $\{A_k\}$ and $\{M_k(x^*)\}$ in terms of $\|x^k-x^*\|^2$. 
Recall definitions \eqref{def:rho}-\eqref{def:M} and \eqref{prop:zk:gk}-\eqref{prop:zk:hk}.

\quad

\begin{prop}[Bounds on increments]\label{prop:A}
Consider Assumptions \ref{existence}-\ref{extra.cte.control}. If Assumption \ref{extra.cte.control}(i) 
holds for some $x^*\in X^*$, then, for all $k\in\mathbb{N}_0$, 	
\begin{eqnarray*}
\Lqnorm{A_{k+1}-A_k|\alg_k}&\le &\left[32\left(1+L\alpha_{k}+\mathsf{H}_{k,p}(x^*)\right)^2 + 
2(8+\rho_k)\right]\mathsf{H}_{k,p}(x^*)^2\|x^k-x^*\|^2\\
&&+\left[32\mathsf{H}_{k,p}(x^*)^2+16+2(8+\rho_k)\right]\mathsf{H}_{k,p}(x^*)^2,
\end{eqnarray*}
$$
\Lqnorm{M_{k+1}(x^*)-M_k(x^*)|\alg_k}\le 
\sqrt{\frac{\mathsf{B}}{\mathsf{A}}}\mathsf{H}_{k,p}(x^*)\left[1+L\alpha_{k}+\mathsf{H}_{k,p}(x^*)\right]^2\|x^k-x^*\|^2
$$
$$
+\sqrt{\frac{\mathsf{B}}{\mathsf{A}}}\mathsf{H}_{k,p}(x^*)
\left[1+L\alpha_{k}+(3+2L\alpha_{k})\mathsf{H}_{k,p}(x^*)+2\mathsf{H}_{k,p}(x^*)^2\right]\|x^k-x^*\|
$$
$$
+\sqrt{\frac{\mathsf{B}}{\mathsf{A}}}\mathsf{H}_{k,p}(x^*)\left[\mathsf{H}_{k,p}(x^*)+\mathsf{H}_{k,p}(x^*)^2\right].
$$
Moreover, if Assumption \ref{extra.cte.control}(iii) holds, then, for all $k\in\mathbb{N}_0$,
\begin{eqnarray*}
\Lqnorm{A_{k+1}-A_k|\alg_k}&\leq &\left(16+\rho_k\right)\alpha_{k}^2\frac{C_p^2\sigma^2}{N_{k,\min}},\\
\Lqnorm{M_{k+1}(x^*)-M_k(x^*)|\alg_k}&\leq & 
\left(1+\mathcal{L}\alpha_{k}\right)\alpha_{k}\frac{C_p\sigma}{\sqrt{N_{k,\min}}}\Vert x^k-x^*\Vert\\
&+&\left(\mathcal{M}+\frac{C_p\sigma}{\sqrt{N_{k,\min}}}\right)\alpha_{k}^2\frac{C_p\sigma}{\sqrt{N_{k,\min}}}.
\end{eqnarray*}
\end{prop}
\begin{proof}
Assume first that Assumption \ref{extra.cte.control}(i) holds. We start with the bound 
on $A_{k+1}-A_k$. Definition \eqref{extra.cte.noise.sample1}, Lemma \ref{lema:indsums}, $x^k\in\alg_k$, 
the independence of $\{\xi^k\}$ and $\alg_k$, and the fact that $(a+b)^2\le2a^2+2b^2$ imply
\begin{equation}\label{prop:A.eq1}
\Lqnorm{\left\|\epsilon^k_1\right\|^2 \Big|\alg_k} = 
\Lpnorm{\left\|\epsilon^k_1\right\| \Big|\alg_k}^2\leq 2\frac{\mathsf{A}}{\mathcal{N}_k}C_p^2\sigma(x^*)^2\,(1+\|x^k-x^*\|^2).	
\end{equation}	
We proceed similarly for a bound of $\epsilon^k_2$ as defined in \eqref{extra.cte.noise.sample2}, 
but with the use of the filtration $\widehat\alg_k$. Lemma \ref{lema:indsums}, $z^k\in\widehat\alg_k$ 
and the independence of $\{\eta^k\}$ and $\widehat\alg_k$ imply 
\begin{equation}\label{prop:A.eq2}
\Lpnorm{\left\|\epsilon^k_{2}\right\|\Big|\widehat\alg_k}\leq
\left(\frac{\mathsf{A}}{\mathcal{N}_k}\right)^{\frac{1}{2}}C_p\sigma(x^*)\,(1+\|z^k-x^*\|).
\end{equation}
We condition 
\eqref{prop:A.eq2} with $\Lpnorm{\Lpnorm{\cdot\Big|\widehat\alg_k}\Big|\alg_k}=\Lpnorm{\cdot\Big|\alg_k}$, 
and then take squares, getting
\begin{equation}\label{prop:A.eq3}
\Lqnorm{\left\|\epsilon^k_2\right\|^2 \Big|\alg_k}=
\Lpnorm{\left\|\epsilon^k_2\right\| \Big|\alg_k}^2\leq 
2\frac{\mathsf{A}}{\mathcal{N}_k}C_p^2\sigma(x^*)^2\,\left(1+\Lpnorm{\|z^k-x^*\|\Big|\alg_k}^2\right).
\end{equation}
Finally we use \eqref{prop:A.eq1}, \eqref{prop:A.eq3}, \eqref{def:A}, Proposition \ref{prop:zk} and relation $(a+b)^2\le2a^2+2b^2$, 
obtaining the required bounds on $A_{k+1}-A_k$.

Now we deal with $M_{k+1}(x^*)-M_k(x^*)$. Definition \eqref{extra.cte.noise.sample2}, 
Lemma \ref{lema:indsums}, $z^k\in\widehat\alg_k$ and the independence of $\{\eta^k\}$ and $\widehat\alg_k$ imply 
\begin{equation}\label{ee9}
\Lpnorm{\langle x^*-z^k,\alpha_k\epsilon^k_2\rangle\Big|\widehat\alg_k}
\leq \Vert z^k-x^*\Vert\alpha_k\sqrt{\frac{\mathsf{B}}{\mathcal{N}_k}}C_p\sigma(x^*)(1+\|z^k-x^*\|).
\end{equation}	
In \eqref{ee9}, we first use $\Lqnorm{\cdot|\widehat\alg_k}\le\Lpnorm{\cdot|\widehat\alg_k}$ and then take $\Lqnorm{\cdot\Big|\alg_k}$, obtaining
\begin{eqnarray}\label{ee10}
\nonumber\Lqnorm{\langle x^*-z^k,\alpha_k\epsilon^k_2\rangle\Big|\alg_k}
&\le & \alpha_{k}\sqrt{\frac{\mathsf{B}}{\mathcal{N}_k}}C_p\sigma(x^*)\,\left(\Lqnorm{\|x^*-z^k\|\Big|\alg_k}+
\Lqnorm{\|x^*-z^k\|^2\Big|\alg_k}\right)\nonumber\\
&\le & \alpha_{k}\sqrt{\frac{\mathsf{B}}{\mathcal{N}_k}}C_p\sigma(x^*)\,\left(\Lpnorm{\|x^*-z^k\|\Big|\alg_k}+
\Lpnorm{\|x^*-z^k\|\Big|\alg_k}^2\right),
\end{eqnarray}	
using the fact that $\Lqnorm{\Lqnorm{\cdot\Big|\widehat\alg_k}\Big|\alg_k}=\Lqnorm{\cdot\Big|\alg_k}$ in the first inequality, and that fact that 
$\Lqnorm{\cdot|\alg_k}\le\Lpnorm{\cdot|\alg_k}$ in the second inequality. Definition \eqref{def:M}, \eqref{ee10}, 
Proposition \ref{prop:zk} and the fact that $(a+b)^2\le2a^2+2b^2$ entail the required bound on $M_{k+1}(x^*)-M_k(x^*)$.

Suppose now that Assumption \ref{extra.cte.control}(iii) hold. First we prove the bound on $\{A_{k+1}-A_k\}$. The proof is similar 
to the previous case, but 
\eqref{prop:A.eq1} and \eqref{prop:A.eq3} are replaced respectively by
\begin{equation}\label{prop:A.eq4}
\Lqnorm{\left\|\epsilon^k_1\right\|^2 \Big|\alg_k}\leq \frac{C_p^2\sigma^2}{N_{k,\min}},
\quad
\Lqnorm{\left\|\epsilon^k_2\right\|^2 \Big|\alg_k}\leq \frac{C_p^2\sigma^2}{N_{k,\min}},
\end{equation}
using Lemma \ref{lema:indsums}. 
From \eqref{def:A} and \eqref{prop:A.eq4} we obtain the required bound on $A_{k+1}-A_k$.
We deal now with $\{M_{k+1}(x^*)-M_k(x^*)\}$. The proof is similar to the previous case, but instead of \eqref{ee9}
now we have
\begin{equation}\label{ee11}
\Lpnorm{\langle x^*-z^k,\alpha_k\epsilon^k_2\rangle\Big|\widehat\alg_k} \leq
\Vert z^k-x^*\Vert\alpha_k\frac{C_p\sigma}{\sqrt{N_{k,\min}}},
\end{equation}
using Lemma \ref{lema:indsums}. In \eqref{ee11}, we use $\Lqnorm{\cdot|\widehat\alg_k}\le\Lpnorm{\cdot|\widehat\alg_k}$ 
and then take $\Lqnorm{\cdot\Big|\alg_k}$, getting
\begin{equation}\label{ee12}
\Lqnorm{\langle x^*-z^k,\alpha_k\epsilon^k_2\rangle\Big|\alg_k}\le 
\alpha_{k}\frac{C_p\sigma}{\sqrt{N_{k,\min}}}\Lqnorm{\Vert z^k-x^*\Vert\Big|\alg_k},
\end{equation}
using the fact that  $\Lqnorm{\Lqnorm{\cdot\Big|\widehat\alg_k}\Big|\alg_k}=\Lqnorm{\cdot\Big|\alg_k}$. The definition \eqref{def:M}, 
Proposition \ref{prop:zk} with $\Lqnorm{\cdot|\alg_k}\le\Lpnorm{\cdot|\alg_k}$ and \eqref{ee12} 
imply the required  bound on $M_{k+1}(x^*)-M_k(x^*)$.
\end{proof}

\quad

Now, we combine Lemma \ref{extra.cte.korpelevich} and Proposition \ref{prop:A} in the following 
recursive relation. Recall definitions \eqref{prop:zk:gk}-\eqref{prop:zk:hk}.

\quad

\begin{prop}[Stochastic quasi-Fej\'er property]\label{extra.cte.fejer}
Consider Assumptions \ref{existence}-\ref{extra.cte.control}. Then for all $k\in\mathbb{N}_0$ and 
for $x^*\in X^*$ as described in Assumption \ref{extra.cte.control}, it holds that
\begin{equation}\label{vvv}
\esp\left[\|x^{k+1}-x^*\|^2|\alg_k\right] \le \|x^{k}-x^*\|^2 -
\frac{\rho_k}{2}r_{\alpha_k}(x^k) + \mathsf{C}_k(x^*)\frac{\mathcal{I}\|x^k-x^*\|^2+1}{\mathcal{N}'_k},
\end{equation}
where, under Assumption \ref{extra.cte.control}(i)-(ii), $\mathcal{I}=1$, $\mathcal{N}_k'=\mathcal{N}_k$ and
\begin{equation}\label{extra.cte.fejer:ck:1}
\mathsf{C}_k(x^*):= \mathsf{A} \mathsf{G}_{k,2}(x^*)^2\left[32\left(1+L\alpha_{k}+\mathsf{H}_{k,2}(x^*)\right)^2 + 18\right],
\end{equation}
while, under Assumption \ref{extra.cte.control}(iii), $\mathcal{I}=0$, $\mathcal{N}_k'=N_{k,\min}$ and
\begin{equation}\label{extra.cte.fejer:ck:2}
\mathsf{C}_k(x^*):=\mathsf{C}_k=(16+\rho_k)\alpha_{k}^2C_2^2\sigma^2.
\end{equation}
\end{prop}
\begin{proof}
We first note that from definition \eqref{def:M}, for any $x^*\in X^*$, $\{M_{k+1}(x^*)-M_k(x^*),\alg_k\}$ defines a  
martingale difference sequence, that is, $\esp[M_{k+1}(x^*)-M_k(x^*)|\alg_k]$ $=0$ for all $k\in\mathbb{N}_0$. 
Indeed, $z^k\in\widehat\alg_k$ and the independence between $\eta^k$ and $\widehat\alg_k$ 
imply that $\esp[\epsilon^k_2 |\widehat\alg_k]=0$. 
This equality and  the fact that $z^k\in\widehat\alg_k$ imply further that
$$
\esp\left[M_{k+1}(x^*)-M_k(x^*)\Big|\widehat\alg_k\right]=2\left\langle x^*-z^k,\alpha_k\esp\left[\epsilon^k_2 \Big|\widehat\alg_k\right]\right\rangle=0.
$$
We take $\esp[\cdot|\alg_k]$ above and use the hereditary property $\esp[\esp[\cdot|\widehat\alg_k]|\alg_k]=\esp[\cdot|\alg_k]$ 
in order to get $\esp[M_{k+1}(x^*)-M_k(x^*)|\alg_k]=0$ as claimed.

We now take the conditional expectation with respect to $\alg_k$ in Lemma \ref{extra.cte.korpelevich}, obtaining
\begin{equation}\label{prop:telescope.eq2}
\esp[\Vert x^{k+1}-x^*\Vert^2|\alg_k]\le\Vert x^{k}-x^*\Vert^2 -\frac{\rho_k}{2}r_{\alpha_k}(x^k)^2+\esp[A_{k+1}-A_k|\alg_k],
\end{equation}
using the facts that $x^k\in\widehat\alg_k$ and that $\{M_{k+1}(x^*)-M_k(x^*),\alg_k\}$ is a martingale difference sequence.
We have that 
$$
32\left(1+L\alpha_{k}+\mathsf{H}_{k,2}(x^*)\right)^2+2(8+\rho_k)>
32 \mathsf{H}_{k,2}(x^*)^2+16+2(8+\rho_k).
$$ 
Hence, under Assumption \ref{extra.cte.control}(i)-(ii), the bound of $\{A_{k+1}-A_{k}\}$ given
in Proposition \ref{prop:A} implies that 
\begin{equation}\label{prop:fejer.eq2}
\Lqnorm{A_{k+1}-A_{k}|\alg_k}\le \mathsf{C}_k(x^*)\frac{\mathcal{I}\|x^k-x^*\|^2+1}{\mathcal{N}'_k},
\end{equation} 
for all $k\ge0$, with $\mathcal{I}=1$, $\mathcal{N}'_k=\mathcal{N}_k$ and definition of $\mathsf{C}_k(x^*)$.

Under Assumption \ref{extra.cte.control}(iii), 
Proposition \ref{prop:A} 
implies \eqref{prop:fejer.eq2}, with $\mathcal{I}=1$ and $\mathcal{N}'_k=N_{k,\min}$ and definition of $\mathsf{C}_k$.
The claimed relation follows from \eqref{prop:telescope.eq2} and \eqref{prop:fejer.eq2} for $q=1$.
\end{proof}

\begin{rem}
\emph{
Under Assumption \ref{extra.cte.control}(i), the inequality of Proposition \ref{extra.cte.fejer} holds for
a given $x^*\in X^*$, as described in Assumption \ref{extra.cte.control}(i). Under Assumption 
\ref{extra.cte.control}(ii) or (iii), the inequality of Prop. \ref{extra.cte.fejer} holds for every $x^*\in X^*$.
}
\end{rem}

\begin{rem}[Bounds of $A_{k+1}-A_k$]\label{estimate.C}
\emph{Recall definition of $\mathsf{C}_k(x^*)$ in the previous proposition. Under Assumption \ref{extra.cte.control}(i)-(ii), the upper bound on 
\begin{equation}\label{def:C}
\mathsf{C}(x^*):=\sup_k \mathsf{C}_k(x^*)
\end{equation} 
depends only on $p$, $L$, the sampling rate $\mathcal{N}_k$, $\sigma(x^*)^2$ and $\hat\alpha$ as defined in 
Assumption \ref{extra.cte.step}. From \eqref{prop:zk:gk} and \eqref{extra.cte.fejer:ck:1}, under 
Assumption \ref{extra.cte.control}(i)-(ii) there exists constant $c>1$ such that
\begin{equation}\label{estimate:C:eq1}
\frac{\mathsf{C}_k(x^*)}{\mathcal{N}_k}\leq c \mathsf{H}_{k,2}(x^*)^2\left(1+\mathsf{H}_{k,2}(x^*)^2\right)
\leq c\hat\alpha^2C_2^2\frac{\mathsf{A}\sigma(x^*)^2}{\mathcal{N}_k}\left(1+\hat\alpha^2C_2^2\frac{\mathsf{A}\sigma(x^*)^2}{\mathcal{N}_k}\right), 
\end{equation}
that is, $\mathsf{C}(x^*)\lesssim \sigma(x^*)^4$. But since at least $\mathcal{N}_k\ge N_{k,\min}\approx \Theta k^{1+a}(\ln k)^{1+b}$, for some $\Theta>0$,
$a>0$, $b\ge-1$ or $a=0$, $b>0$, the following non-asymptotic bound holds:
\begin{equation}\label{prop:fejer.eq3}
\mathsf{C}_k(x^*)\lesssim \sigma(x^*)^2\left(1+\frac{\sigma(x^*)^2}{\Theta k^{1+a}(\ln k)^{1+b}}\right),
\end{equation}
which is $\approx\sigma(x^*)^2$ for an iteration index $k$ large enough as compared to $\sigma(x^*)^2$.\footnote{In terms of numerical constants, a sharper bound can be obtained by exploiting 
the first order term $H_{k,2}(x^*)\sim \sigma(x^*)\mathcal{N}_k^{-1/2}$ in 
the definition of $\mathsf{C}_k(x^*)$. We do not carry out this procedure here.} 
Under Assumption \ref{extra.cte.control}(iii), the following \emph{uniform} bound holds on $X^*$:
$
\mathsf{C}_k\lesssim\sigma^2.
$
}
\end{rem}

\quad

We finish this section with the asymptotic convergence result.

\quad

\begin{thm}[Asymptotic convergence]
\label{extra.cte.convergence} 
Under Assumptions \ref{existence}-\ref{extra.cte.control}, a.s.
the sequence $\{x^k\}$ generated by \eqref{algorithm.extra.cte1}-\eqref{algorithm.extra.cte2} 
is bounded, 
$
\lim_{k\rightarrow\infty}\dist(x^k,X^*)=0,
$
and
$
r_{\alpha_k}(x^k)
$ 
converges to $0$ almost surely and in $L^2$.
In particular, a.s. every cluster point of $\{x^k\}$ belongs to $X^*$.
\end{thm}
\begin{proof}
The result in Proposition \ref{extra.cte.fejer} may be rewritten as 
\begin{equation}\label{teo:conv:eq1}
\esp\left[\|x^{k+1}-x^*\|^2|\alg_k\right] \le
\left(1+\frac{\mathcal{I} \mathsf{C}(x^*)}{\mathcal{N}_k'}\right)\|x^{k}-x^*\|^2 -\frac{\rho_k}{2}r_{\alpha_k}(x^k)^2 + 
\frac{\mathsf{C}(x^*)}{\mathcal{N}_k'},
\end{equation}
for all $k\ge0$ and for some $x^*\in X^*$, as ensured by Assumption \ref{extra.cte.control}. 
Taking into account 
Assumption \ref{extra.cte.samples}, 
i.e., $\sum_k\mathcal{N}_k^{-1}<\infty$, 
\eqref{teo:conv:eq1}
and the fact that $x^k\in\alg_k$, we apply Theorem \ref{rob} 
with $y_k:=\Vert x^k-x^*\Vert^2$, $a_k=\mathcal{I}\cdot \mathsf{C}(x^*)/\mathcal{N}_k'$, $b_k=\mathsf{C}(x^*)/\mathcal{N}_k'$ 
and $u_k:=\rho_k r_{\alpha_k}(x^k)^2/2$, in order to conclude that a.s.
$\{\Vert x^k-x^*\Vert^2\}$ converges and $\sum_k\rho_k r_{\alpha_k}(x^k)^2<\infty$. In particular, $\{x^k\}$ is bounded, and 
$
\hat\rho\sum_k r_{\alpha_k}(x^k)^2\le\sum_k\rho_k r_{\alpha_k}(x^k)^2<\infty,
$
where $\hat\rho:=1-6\hat\alpha^2L^2>0$ and $\hat\alpha:=\sup_k\alpha_{k}$ by Assumption \ref{extra.cte.step}. 
Hence, almost surely, 
\begin{equation}\label{teo:conv:eq2}
0 = \lim_{k\rightarrow\infty}r_{\alpha_k}(x^k)^2\\
=\lim_{k\rightarrow\infty}\left\Vert x^k-\Pi\left[x^k-\alpha_k T(x^k)\right]\right\Vert^2.
\end{equation}
The boundedness of the stepsize sequence, 
\eqref{teo:conv:eq2}, 
and the continuity of $T$ (Assumption \ref{extra.cte.lipschitz}) and $\Pi$ (Lemma \ref{proj}(iii)) imply that a.s. every  cluster point $\bar x$ of $\{x^k\}$ satisfies
$$
0=\bar x-\Pi\left[\bar x-\bar\alpha T(\bar x)\right],
$$
for some $\bar\alpha>0$, in view of Assumption \ref{extra.cte.step}: i.e. the fact that the stepsizes are
bounded away from zero; from Lemma \ref{proj}(iv) we have that $\bar x\in X^*$. Almost surely, the boundedness of $\{x^k\}$ and 
the fact that every cluster point of $\{x^k\}$ belongs to $X^*$ 
imply that $\lim_{k\rightarrow\infty}\dist(x^k,X^*)=0$ as claimed. 

We now prove convergence of $r_{\alpha_k}(x^k)$ to $0$ in $L^2$. We take total expectation in \eqref{teo:conv:eq1} and obtain for all $k\ge0$,
\begin{equation}\label{teo:conv:eq1.1}
\esp\left[\|x^{k+1}-x^*\|^2\right] \le
\left(1+\frac{\mathcal{I} \mathsf{C}(x^*)}{\mathcal{N}_k'}\right)\esp\left[\|x^{k}-x^*\|^2\right] -\frac{\rho_k}{2}\esp\left[r_{\alpha_k}(x^k)^2\right] + 
\frac{\mathsf{C}(x^*)}{\mathcal{N}_k'}.
\end{equation}
Taking into account 
Assumption \ref{extra.cte.samples}, 
i.e., $\sum_k\mathcal{N}_k^{-1}<\infty$, 
\eqref{teo:conv:eq1.1}
we apply Theorem \ref{rob} 
with $y_k:=\esp[\Vert x^k-x^*\Vert^2]$, $a_k=\mathcal{I}\cdot \mathsf{C}(x^*)/\mathcal{N}_k'$, $b_k=\mathsf{C}(x^*)/\mathcal{N}_k'$ 
and $u_k:=\rho_k\esp[r_{\alpha_k}(x^k)^2]/2$, in order to conclude that 
$\sum_k\rho_k\esp[r_{\alpha_k}(x^k)^2]<\infty$. In particular, 
$
\hat\rho\sum_k \esp[r_{\alpha_k}(x^k)^2]\le\sum_k\rho_k \esp[r_{\alpha_k}(x^k)^2]<\infty,
$
which implies that $\lim_{k\rightarrow\infty}\esp[r_{\alpha_k}(x^k)^2]=0$ as claimed.
\end{proof}

\subsection{Convergence rate and complexity analysis}\label{ss3.4}
We now study the convergence rate and the oracle complexity of our algorithm. 
Besides the relation in Proposition \ref{extra.cte.fejer} for $p=2$,
we can also obtain a recursive relation for higher order moments, 
assuming that $p\ge4$. This recursion, derived as consequence of Propositions \ref{prop:A} and \ref{prop:telescopeandmartingale}(i), will give an explicit upper bound on the $p$-norm of the generated sequence 
(see Proposition \ref{prop:lp:bound}). The explicit bound of the $2$-norm of the sequence will be used for giving explicit estimates on the convergence rate and complexity under Assumption \ref{extra.cte.control}(i)-(ii), 
i.e., when $X$ and $T$ are unbounded, in Theorem \ref{thm:convergence:rate}. In this setting, we will also obtain sharper 
estimates of the constants assuming uniform variance over the \emph{solution set} 
(see Propositions \ref{prop:telescopeandmartingale}(ii), Proposition \ref{prop:lp:bound}(ii) 
and Theorem \ref{thm:convergence:rate:unif}). Important cases satisfying these assumptions include 
the cases in which $X^*$ is a singleton or a compact set (which can occur even when the feasible set $X$ is unbounded).\footnote{This occurs when the solution set is a singleton in the case of a strictly or strongly pseudo-monotone operator. See Theorems 2.3.5 and 2.3.16 in \cite{facchinei} for general conditions of compactness of the solution set of a pseudo-monotone VI. An example is the so called \emph{strictly feasible} complementarity problem over a cone.} Under the stronger Assumption \ref{extra.cte.control}(iii), 
that is, uniform variance over the feasible set, even sharper bounds on the estimates will be presented 
(see Propositions \ref{prop:telescopeandmartingale}(iii) and \ref{prop:lp:bound}(iii) and Theorem \ref{thm:convergence:rate:unif}). 

\begin{prop}[Improved stochastic quasi-Fej\'er properties]\label{prop:telescopeandmartingale}
\begin{itemize}
\item[i)] If Assumption \ref{extra.cte.control}(i) holds for $p\ge4$ and 
some $x^*\in X^*$, then for all $k_0,k$ such that $0\leq k_0<k$,
it holds that
\begin{eqnarray*}
\Lqnorm{\Vert x^k-x^*\Vert^2}&\leq &\Lqnorm{\Vert x^{k_0}-x^*\Vert^2}+\\ 
&& C_q\,\sqrt{\sum_{i=k_0+1}^k\Lqnorm{M_{i}(x^*)-M_{i-1}(x^*)}^2} + \sum_{i=k_0+1}^k\Lqnorm{A_i-A_{i-1}}.
\end{eqnarray*}
\item[ii)] If Assumption \ref{extra.cte.control}(ii) holds for $p\ge2$ then $\mathsf{C}:=\sup_k \mathsf{C}_k:X^*\rightarrow\re_+$, as defined in \eqref{extra.cte.fejer:ck:1} and \eqref{def:C}, is a 
locally bounded measurable function, and for all $k\ge0$,
$$
\esp\left[\dist(x^{k+1},X^*)^2|\alg_k\right] \le \dist(x^{k},X^*)^2 -\frac{\rho_k}{2}r_{\alpha_k}(x^k)^2 + 
\mathsf{C}_k\left(\Pi_{X^*}(x^k)\right)\frac{\dist(x^{k},X^*)^2+1}{\mathcal{N}_k}.
$$
\item[iii)] If Assumption \ref{extra.cte.control}(iii) holds then for all $k\ge0$,
\begin{equation}\label{www}
\esp\left[\dist(x^{k+1},X^*)^2|\alg_k\right] \le \dist(x^{k},X^*)^2 -
\frac{\rho_k}{2}r_{\alpha_k}(x^k)^2 + \frac{17C_2^2\hat\alpha^2\sigma^2}{N_{k,\min}}.
\end{equation}
\end{itemize}
\end{prop}
\begin{proof}
i) Define for simplicity $d_k:=\Vert x^k-x^*\Vert^2$. Summing relation in Lemma \ref{extra.cte.korpelevich} from $k_0$ to $k-1$ we obtain
$
0\le d_{k}\le d_{k_0} + M_{k}(x^*)-M_{k_0}(x^*)+A_{k}-A_{k_0}.
$
This relation implies
\begin{equation}\label{ee20}
0\le d_{k}\le d_{k_0} + [M_{k}(x^*)-M_{k_0}(x^*)+A_{k}-A_{k_0}]_+,
\end{equation}
since $a\le b\Rightarrow[a]_+\le[b]_+$ for any $a,b\in\re$. We take the $q$-norm in \eqref{ee20}, getting
\begin{eqnarray}\label{ee15}
\nonumber \Lqnorm{d_k}&\leq &\Lqnorm{d_{k_0}}+ \Lqnorm{[M_k(x^*)-M_{k_0}(x^*) + A_k-A_{k_0}]_+} \\
\nonumber &\leq& \Lqnorm{d_{k_0}} + \Lqnorm{M_k(x^*)-M_{k_0}(x^*) + A_k-A_{k_0}}\\ 
&\leq &\Lqnorm{d_{k_0}} + \Lqnorm{M_{k}(x^*)-M_{k_0}(x^*)} + \sum_{i=k_0+1}^k\Lqnorm{A_i-A_{i-1}},
\end{eqnarray}
using Minkowski's inequality in the first and last inequalities and the fact that $\Lqnorm{U_+}\le\Lqnorm{U}$ 
for any random variable $U$ in the second inequality. 

Since $q\ge2$ ($p\ge4$), the norm of the martingale term above may be estimated via the BDG inequality \eqref{BDG}
applied to the martingale $\tilde{M}_i:=M_{k_0+i}(x^*)-M_{k_0}(x^*)$. This gives:
\begin{equation}\label{ee16}
\Lqnorm{M_{k}(x^*)-M_{k_0}(x^*)}\leq C_q\,\sqrt{\sum_{i=k_0+1}^k\Lqnorm{M_{i}(x^*)-M_{i-1}(x^*)}^2}.
\end{equation}
Plugging \eqref{ee16}  into \eqref{ee15} completes the proof of item (i).

ii) Under Assumption \ref{extra.cte.control}(ii), we define $\bar x^k:=\Pi_{X^*}(x^k)$, recalling Assumption \ref{existence},
 and obtain from Proposition \ref{extra.cte.fejer}:
\begin{eqnarray*}
\esp\big[\dist(x^{k+1},X^*)^2|\alg_k\big] &\le &
\esp\big[\Vert x^{k+1}-\bar x^k\Vert^2\big|\alg_k\big]\\
&\le &\Vert x^{k}-\bar x^k\Vert^2-\frac{\rho_k}{2}r_{\alpha_k}(x^k)^2+ \mathsf{C}_k(\bar x^k)\frac{\|x^k-\bar x^k\|^2+1}{\mathcal{N}_k}\\
&=&\dist(x^k,X^*)^2-\frac{\rho_k}{2}r_{\alpha_k}(x^k)^2+
\mathsf{C}_k\left(\Pi_{X^*}(x^k)\right)\frac{\dist(x^k,X^*)^2+1}{\mathcal{N}_k},
\end{eqnarray*}
using the fact that $\bar x^k\in X^*$ in the first inequality, the facts that 
$\mathsf{C}_k(\bar x^k)\in\alg_k$ (which holds because $x^k\in\alg_k$, $\Pi_{X^*}$ is continuous 
and $\mathsf{C}_k$ is measurable), and $\bar x^k\in X^*$ (cf. Proposition \ref{extra.cte.fejer})
in the second inequality, and the fact that $\dist(x^k,X^*)=\Vert x^{k}-\bar x^k\Vert$ 
in the equality. Note that the function $\mathsf{C}:X^*\rightarrow\re_+$ is measurable and
 locally bounded by Assumption \ref{extra.cte.control}(ii) and the definition of $\mathsf{C}_k(x^*)$.

iii) We use a proof line analogous to the one in item (ii), with Assumption \ref{extra.cte.control}(iii) and Proposition \ref{extra.cte.fejer}.
\end{proof}

\quad

The following result gives explicit bounds on the $p$-norm of the sequence in the unbounded setting. 
In order to make the presentation easier, we introduce some 
definitions. Recall the constant $c$ defined in Remark \ref{estimate.C}. We set 
\begin{eqnarray}
\mathsf{D}_p(x^*)&:=&2c\hat\alpha^2 C_p^2\sigma(x^*)^2,\label{def:D}
\end{eqnarray}
with $\mathsf{D}(x^*):=\mathsf{D}_2(x^*)$. Define also $\mathsf{B}_2(x^*):=0$ and for $p\ge4$, set $\mathsf{\tilde G}_p(x^*):=C_p\hat\alpha\sigma(x^*)$ and
\begin{equation}
\mathsf{B}_p(x^*):=\sqrt{3\mathsf{B}}C_q\mathsf{\tilde G}_p(x^*)\left[(1+L\hat\alpha)^2+(3+2L\hat\alpha)\sqrt{\mathsf{A}}\mathsf{\tilde G}_p(x^*)+2\mathsf{A}\mathsf{\tilde G}_p(x^*)^2\right].\label{prop:def:Bp}
\end{equation}
\begin{prop}[Uniform boundedness in $L^p$]\label{prop:lp:bound}
\begin{itemize}
\item[i)] Let Assumptions \ref{existence}-\ref{extra.cte.control}(i) hold for some $x^*\in X^*$ and $p\in\{2\}\cup[4,\infty)$. Choose $k_0:=k_0(x^*)\in\mathbb{N}$ and $\gamma:=\gamma(x^*)>0$ such that 
\begin{equation}\label{prop:lp:bound:beta}
\beta(x^*):=\mathsf{B}_p(x^*)\sqrt{\gamma}+\mathsf{D}_p(x^*)\gamma+\mathsf{D}_p(x^*)^2\gamma^2<1,\quad\sum_{k\ge k_0}\frac{1}{\mathcal{N}_k}<\gamma. 
\end{equation}
Then
$$
\sup_{k\ge k_0}\Lpnorm{\Vert x^k-x^*\Vert}^2\le\mathsf{c}_p(x^*)\left[1+\Lpnorm{\Vert x^{k_0}-x^*\Vert}^2\right],
$$
with $\mathsf{c}_2(x^*)=[1-\beta(x^*)]^{-1}$ and $\mathsf{c}_p(x^*)=4[1-\beta(x^*)]^{-2}$ for $p\ge4$.

\item[ii)] Let Assumptions \ref{existence}-\ref{extra.cte.control}(ii) hold and suppose there exists $\sigma>0$ such that $\sigma(x^*)\le\sigma$ for all $x^*\in X^*$. Let $\phi\in(0,\frac{\sqrt{5}-1}{2})$.
Choose $k_0\in\mathbb{N}$ such that 
$
\sum_{k\ge k_0}\frac{1}{\mathcal{N}_k}\le\frac{\phi}{2c\hat\alpha^2 C_2^2\sigma^2}.
$
Then 
$$
\sup_{k\ge k_0}\esp\left[\dist(x^k,X^*)^2\right]\le\frac{1+\esp\left[\dist(x^{k_0},X^*)^2\right]}{1-\phi-\phi^2}.
$$

\item[iii)] If Assumptions \ref{existence}-\ref{extra.cte.control}(iii) hold then 
$$
\sup_{k\ge0}\esp\left[\dist(x^k,X^*)^2\right]\le\dist(x^0,X^*)^2+\sum_{k=0}^\infty
\frac{17C_2^2\hat\alpha^2\sigma^2}{N_{k,\min}}.
$$
\end{itemize}
\end{prop}
\begin{proof}
i) Denote $d^k:=\|x^k-x^*\|$. 
We first unify the Fej\'er-type relations obtained so far under Assumption \ref{extra.cte.control}(i)-(ii) as: for all $k>k_0$,
\begin{eqnarray}
\Lpnorm{d_k}^2&\le &\Lpnorm{d_{k_0}}^2+\mathsf{B}_p(x^*)\sqrt{\sum_{i=k_0}^{k-1}\frac{1+\Lpnorm{d_i}^2+
\Lpnorm{d_i}^4}{\mathcal{N}_i}}+\nonumber\\
&+& \mathsf{D}_p(x^*)\sum_{i=k_0}^{k-1}\frac{1+\Lpnorm{d_i}^2}{\mathcal{N}_i}+
\mathsf{D}_p(x^*)^2\sum_{i=k_0}^{k-1}\frac{1+\Lpnorm{d_i}^2}{\mathcal{N}_i^2}.\label{prop:lp:bound:eq1}
\end{eqnarray}
Indeed, for $p=2$, we have $\mathsf{B}_2(x^*)=0$ so that \eqref{prop:lp:bound:eq1} results by summing the  
relation in Proposition \ref{extra.cte.fejer} 
from $k_0$ to $k-1$ 
and using the estimate in \eqref{estimate:C:eq1}, the fact that $\mathsf{A}\le2$, $c>1$ and the 
definition of $\mathsf{D}_2(x^*)$ as stated before the proposition. For $p\ge4$, we recall the bounds of 
increments of $\{M_k(x^*)\}$ in Proposition \ref{prop:A}. The common factor is bounded
by $\sqrt{\mathsf{B}/\mathsf{A}}\cdot\mathsf{H}_{k,p}(x^*)\le\sqrt{\mathsf{B}}\mathsf{\tilde G}_p(x^*)/\sqrt{\mathcal{N}_k}$. 
Using the definition of $\mathsf{H}_{k,p}(x^*)$ in \eqref{prop:zk:hk}, $\hat\alpha$ in Assumption \ref{extra.cte.step}, 
$\mathsf{\tilde G}_p(x^*)$ and $\mathcal{N}_k\ge1$, it is easy to see that, in the bound of $M_{k+1}(x^*)-M_k(x^*)$ 
in Proposition \ref{prop:A}, the sum of terms multiplying $\sqrt{\mathsf{B}/\mathsf{A}}\cdot\mathsf{H}_{k,p}(x^*)$ 
is at most $(1+L\hat\alpha)^2+(3+2L\hat\alpha)\sqrt{\mathsf{A}}\mathsf{\tilde G}_p+2\mathsf{A}\mathsf{\tilde G}_p^2$. 
We use these bounds, the facts that $(\Lpnorm{d_i}^2+\Lpnorm{d_i}+1)^2\le3(\Lpnorm{d_i}^4+\Lpnorm{d_i}^2+1)$ and the 
definition \eqref{prop:def:Bp} in order to obtain, for all $i\in\mathbb{N}_0$,	
\begin{equation}\label{prop:lp:bound:eq2}
\Lqnorm{M_{i+1}(x^*)-M_i(x^*)}^2\le \mathsf{B}_p(x^*)^2\frac{1+\Lpnorm{d_i}^2+\Lpnorm{d_i}^4}{\mathcal{N}_i}.
\end{equation}
The proof of \eqref{prop:lp:bound:eq1} for $p\ge4$ follows 
from \eqref{prop:fejer.eq2}, \eqref{estimate:C:eq1} with $\mathsf{A}\le2$, $c>1$ and the 
definition of $\mathsf{D}_p(x^*)$ as well as \eqref{prop:lp:bound:eq2} and Proposition \ref{prop:telescopeandmartingale}(i).

By Assumption \ref{extra.cte.samples}, we can choose $k_0\in\mathbb{N}_0$ and $\gamma>0$ as 
in \eqref{prop:lp:bound:beta}. In particular, $\sum_{i\geq k_0}\mathcal{N}_i^{-2}<\gamma^2$. 
Given an arbitrary $a>\Lpnorm{d_{k_0}}$, define:
$
\tau_{a}:=\inf\{k> k_0:\Lpnorm{d_{k}}\ge a\}.	
$
Suppose first that $\tau_a<\infty$ for all $a>\Lpnorm{d_{k_0}}$. By \eqref{prop:lp:bound:beta}, \eqref{prop:lp:bound:eq1} and 
the definition of $\tau_a$, we have 
\begin{eqnarray}
a^2\le\Lpnorm{d_{\tau_a}}^2\le \Lpnorm{d_{k_0}}^2+
\mathsf{B}_p(x^*)\sqrt{\sum_{i=k_0}^{\tau_a-1}\frac{1+a^2+a^4}{\mathcal{N}_i}}+&&\label{prop:lp:bound:eq4}\\
+\mathsf{D}_p(x^*)\sum_{i=k_0}^{\tau_a-1}\frac{a^2+1}{\mathcal{N}_i}+
\mathsf{D}_p(x^*)^2\sum_{i=k_0}^{\tau_a-1}\frac{a^2+1}{\mathcal{N}_i^2}&&\nonumber\\
\le \Lpnorm{d_{k_0}}^2+\mathsf{B}_p(x^*)\sqrt{\gamma}\left(1+a+a^2\right)
+\mathsf{D}_p(x^*)\gamma\left(1+a^2\right)+
\mathsf{D}_p(x^*)^2\gamma^2\left(1+a^2\right).&&\nonumber
\end{eqnarray}
For $p=2$, $\mathsf{B}_2(x^*)=0$. By \eqref{prop:lp:bound:eq4}, taking $\beta:=\beta(x^*)\in(0,1)$ in \eqref{prop:lp:bound:beta}, we get
\begin{equation}\label{prop:lp:bound:eq5}
a^2\le\frac{\Lpnorm{d_{k_0}}^2+1}{1-\beta}.
\end{equation}
For $p\ge4$, by \eqref{prop:lp:bound:eq4}, taking $\beta:=\beta(x^*)$ in \eqref{prop:lp:bound:beta}, we obtain
$
\lambda a^2\le\Lpnorm{d_{k_0}}^2+a+1,
$
with $\lambda:=1-\beta$. It follows that
$$
\left(a-\frac{1}{2\lambda}\right)^2\le\frac{4\lambda\Lpnorm{d_{k_0}}^2+4\lambda+1}{4\lambda^2}
\Longrightarrow
a\le\frac{2\Lpnorm{d_{k_0}}+\sqrt{5}+1}{2\lambda}\le\frac{\Lpnorm{d_{k_0}}+2}{\lambda},
$$
and finally
\begin{equation}\label{prop:lp:bound:eq6}
a^2\le4\frac{\Lpnorm{d_{k_0}}^2+1}{(1-\beta)^2}.
\end{equation}
Since \eqref{prop:lp:bound:eq5}-\eqref{prop:lp:bound:eq6} hold for an arbitrary $a>\Lpnorm{d_{k_0}}$ and $\beta\in(0,1)$, it follows that
$
\sup_{k\ge k_0}\Lpnorm{d_k}^2\le \mathsf{c}_p(x^*)\left[1+\Lpnorm{d_{k_0}}^2\right],
$
with $\mathsf{c}_p(x^*)$ as in the statement of the proposition. This contradicts the initial assumption 
that $\tau_a<\infty$ for all $a>\Lpnorm{d_{k_0}}$. Hence there exists $\bar a>\Lpnorm{d_{k_0}}$ such 
that $\hat a:=\sup_{k\ge k_0}\Lpnorm{d_k}\le\bar a<\infty$ by the definition of $\tau_{\bar a}$. 
For any $k>k_0$, we use the fact that $\Lpnorm{d_{i}}\le\hat a$ for $k_0\le i<k$ in \eqref{prop:lp:bound:eq1}, obtaining
\begin{equation}\label{uuu}
\Lpnorm{d_k}^2\le \Lpnorm{d_{k_0}}^2+\mathsf{B}_p(x^*)\sqrt{\gamma}\left(1+\hat a+\hat a^2\right)
+\mathsf{D}_p(x^*)\gamma\left(1+\hat a^2\right)+
\mathsf{D}_p(x^*)^2\gamma^2\left(1+\hat a^2\right).
\end{equation}
The inequality \eqref{uuu} holds trivially for $k:=k_0$. Thus, after taking the supremum over $k\ge k_0$ in \eqref{uuu},  
we proceed as we did immediately after inequality \eqref{prop:lp:bound:eq4}, 
obtaining \eqref{prop:lp:bound:eq5} and \eqref{prop:lp:bound:eq6}, respectively for $p=2$ and $p\ge4$, 
but with $\hat a$ susbtituting for $a$. 
Using the definition of $\mathsf{c}_p(x^*)$, the claim follows.

ii): The proof line is the same as for the case $p=2$ in item (i), but summing \eqref{vvv} 
with the estimate \eqref{estimate:C:eq1}, 
which gives the following uniform estimate: for all $k\ge0$,
$
\mathsf{C}_k\left(\Pi\left(\bar x^k\right)\right)\mathcal{N}_k^{-1}\leq 
2c\hat\alpha^2 C_2^2 \sigma^2\mathcal{N}_k^{-1}\left(1+2\hat\alpha^2 C_2^2\sigma^2\mathcal{N}_k^{-1}\right).
$
We remark that we may replace $\beta(x^*)$ in \eqref{prop:lp:bound:beta} and \eqref{prop:lp:bound:eq5} 
by $\beta:=2c\hat\alpha^2 C_2^2\sigma^2+4c\hat\alpha^4 C_2^4\sigma^4$. 
In this case, the definition of $\phi$ and $k_0$ imply that $0<1-\phi-\phi^2\le1-\beta$.

iii): Given $k\in\mathbb{N}$, we take total expectation in \eqref{www} and sum from $0$ to $k$, obtaining
$$
\esp\left[\dist(x^{k+1},X^*)^2\right]\le\dist(x^0,X^*)^2+
\sum_{i=0}^k\frac{17C_2^2\hat\alpha^2\sigma^2}{N_{i,\min}}\le\dist(x^0,X^*)^2+
\sum_{i=0}^\infty\frac{17C_2^2\hat\alpha^2\sigma^2}{N_{i,\min}},
$$
and the claim follows.
\end{proof}
\begin{rem}\label{rem:prop:lp:bound}
\emph{In statement of Proposition \ref{prop:lp:bound}(i), for $p\ge2$, it is sufficient to 
set $\phi\in(0,\frac{\sqrt{5}-1}{2})$ and $k_0\in\mathbb{N}_0$ such that
$
\sum_{k\ge k_0}\mathcal{N}_k^{-1}\le\phi\mathsf{D}(x^*)^{-1}
$
to obtain 
$
\sup_{k\ge k_0}\esp[\Vert x^k-x^*\Vert^2]\le\frac{1+\esp[\Vert x^{k_0}-x^*\Vert^2]}{1-\phi-\phi^2}.
$}
\end{rem}

\quad

We now give explicit estimates on the convergence rate and oracle complexity. In the sequel 
we assume that the stepsize sequence is constant. Proposition 10.3.6 in \cite{facchinei} states that $\{r_a:a>0\}$ is a 
family of equivalent merit functions of VI$(T,X)$. Hence, the convergence rate analysis can be deduced for 
varying stepsizes satisfying Assumption \ref{extra.cte.step}, and constant stepsizes are assumed just for 
simplicity. Recall definition of $\mathsf{D}(x^*)$ in \eqref{def:D}. Define, for $k,\ell\in\mathbb{N}_0\cup\{\infty\}$, $\phi\in\re$, 
$x^*\in X^*$ and $\alpha\in\left(0,1/\sqrt{6}L\right)$,
\begin{eqnarray}
\rho:=1-6\alpha^2L^2,\quad\quad\mathsf{a}_0^k:=\sum_{i=0}^k\frac{1}{\mathcal{N}_i},\quad\quad
\mathsf{b}_0^k:=\sum_{i=0}^k\frac{1}{\mathcal{N}_i^2},\label{def:ak:bk}\\
\mathsf{J}(x^*,k,\phi):=\frac{1+\max_{0\le i\le k}\esp[\Vert x^{i}-x^*\Vert^2]}{1-\phi-\phi^2},\label{def:Jx}\\
\mathsf{Q}_k(x^*,\ell,\phi):=\frac{2}{\rho}\left\{\Vert x^0-x^*\Vert^2+
\left[1+\mathsf{J}(x^*,\ell,\phi)\right]\left[\mathsf{D}(x^*)\mathsf{a}^k_0+\mathsf{D}(x^*)^2\mathsf{b}^k_0\right]\right\}.\label{def:Qx}
\end{eqnarray}

\begin{thm}[Convergence rate: non-uniform variance]
\label{thm:convergence:rate}
Consider Assumptions \ref{existence}-\ref{extra.cte.control}(i) for some $x^*\in X^*$ and
take $\alpha_k\equiv\alpha\in(0,1/\sqrt{6}L)$. Choose $\phi\in(0,\frac{\sqrt{5}-1}{2})$ and $k_0:=k_0(x^*)\in{\mathbb{N}}$ such that: 
\begin{equation}\label{ee23}
\sum_{k\ge k_0}\frac{1}{\mathcal{N}_k}\le\frac{\phi}{\mathsf{D}(x^*)},
\end{equation} 
where $\mathsf{D}(x^*)$ is defined in \eqref{def:D}. Then for all $\epsilon>0$ there exists $K_\epsilon\in\mathbb{N}$ such that
$$
\esp[r_\alpha(x^{K_\epsilon})^2]\le\epsilon
\le \frac{\mathsf{Q}_\infty(x^*,k_0(x^*),\phi)}{K_\epsilon}.
$$

Additionally, if Assumption \ref{extra.cte.control}(ii) holds then $K_\epsilon$ is independent of $x^*\in X^*$.
\end{thm}
\begin{proof}
First note that finiteness of $\mathsf{a}^\infty_0,\mathsf{b}^\infty_0$ as defined before the statement of the theorem follows from Assumption \ref{extra.cte.samples}, which also ensures existence of $k_0(x^*)$ satisfying \eqref{ee23}, because $\sum_{i\ge k}\mathcal{N}_i^{-1}\rightarrow0$ as $k\rightarrow\infty$. Set $\phi$ and $k_0:=k_0(x^*)$ such that \eqref{ee23} holds. Proposition \ref{prop:lp:bound}(i) for $p=2$ and Remark \ref{rem:prop:lp:bound} imply 
$$
\sup_{k\ge k_0}\esp[\Vert x^k-x^*\Vert^2]\le\frac{1+\esp[\Vert x^{k_0}-x^*\Vert^2]}{1-\phi-\phi^2}\le\frac{1+\max_{0\le k\le k_0}\esp[\Vert x^{k}-x^*\Vert^2]}{1-\phi-\phi^2}=\mathsf{J}(x^*,k_0,\phi).
$$
From the above inequality and $1-\phi-\phi^2\in(0,1)$ we get the following uniform bound:
\begin{equation}\label{thm:conv:rate:eq0}
\sup_{k\ge0}\esp[\Vert x^k-x^*\Vert^2]\le \mathsf{J}(x^*,k_0(x^*),\phi). 
\end{equation}

We now invoke Proposition \ref{extra.cte.fejer}. Given $0\le i\le k$, we take total expectation in \eqref{vvv}, 
with the estimate \eqref{estimate:C:eq1} using $\mathsf{A}\le2$, $c>1$ defined in Remark \ref{estimate.C},  
the definition of $\hat\alpha$ in Assumption \ref{extra.cte.step} and the definition of $\mathsf{D}(x^*)$ in \eqref{def:D}. 
We then sum iteratively with $i$ running from $0$ to $k$, obtaining:
\begin{eqnarray}
&&\frac{\rho}{2}\sum_{i=0}^{k}\esp[r_\alpha(x^i)^2]\nonumber\\
&&\le\Vert x^0-x^*\|^2+\mathsf{D}(x^*)\sum_{i=0}^{k} \frac{1+\esp\left[\|x^i-x^*\|^2\right]}{\mathcal{N}_i}+\mathsf{D}(x^*)^2\sum_{i=0}^{k} \frac{1+\esp\left[\|x^i-x^*\|^2\right]}{\mathcal{N}_i^2}\nonumber\\
&&\le\Vert x^0-x^*\Vert^2+\left(1+\sup_{0\le i\le k}\esp[\Vert x^i-x^*\Vert^2]\right)
\left(\mathsf{D}(x^*)\sum_{i=0}^{k}\frac{1}{\mathcal{N}_i}+
\mathsf{D}(x^*)^2\sum_{i=0}^{k}\frac{1}{\mathcal{N}_i^2}\right)\nonumber\\
&&\le\Vert x^0-x^*\Vert^2+\left[1+\mathsf{J}(x^*,k_0(x^*),\phi)\right]\left[\mathsf{D}(x^*)\mathsf{a}^k_0+
\mathsf{D}(x^*)^2\mathsf{b}^k_0\right]=\frac{\rho}{2}\mathsf{Q}_k(x^*,k_0(x^*),\phi),\nonumber\\
&&\label{thm:conv:rate:eq1}
\end{eqnarray}
using \eqref{thm:conv:rate:eq0} in the last inequality.

Given $\epsilon>0$, define $K=K_\epsilon:=\inf\{k\in\mathbb{N}_0:\esp[r_\alpha(x^k)^2]\le\epsilon\}$. For every $k<K$ we have
\begin{equation}\label{thm:conv:rate:eq2}
\frac{\rho}{2}\epsilon(k+1)<\frac{\rho}{2}\sum_{i=0}^{k}\esp[r_\alpha(x^i)^2],	
\end{equation}
using the fact that $\esp[r_\alpha(x^i)^2]>\epsilon$ for all $0\le i\le k$, which follows from the definition of $K$.

We claim that $K$ is finite. Indeed, if $K=\infty$, then \eqref{thm:conv:rate:eq1} and \eqref{thm:conv:rate:eq2} 
hold for all $k\in\mathbb{N}$. 
Hence, we arrive at a contradiction by letting $k\rightarrow\infty$ and using the facts that $\mathsf{a}_0^\infty<\infty$ 
and $\mathsf{b}_0^\infty<\infty$, which hold by Assumption \ref{extra.cte.samples}. Since $K$ is finite, 
we have $\esp[r_\alpha(x^K)^2]\le\epsilon$ by definition. Setting $k:=M-1$ in \eqref{thm:conv:rate:eq1}-\eqref{thm:conv:rate:eq2}, we get	
$
K\le\frac{\mathsf{Q}_{K-1}(x^*,k_0,\phi)}{\epsilon}\le
\frac{\mathsf{Q}_{\infty}(x^*,k_0,\phi)}{\epsilon},
$
using the definition of $\mathsf{Q}_k(x^*,k_0,\phi)$. We have thus proved the claim. 
Under Assumption \ref{extra.cte.control}(ii), the proof is valid for any $x^*\in X^*$ and, hence, $K$ is independent of $x^*\in X^*$. 
\end{proof}

\quad

In the previous theorem, given $x^*\in X^*$, the constant $\mathsf{Q}_\infty(x^*,k_0(x^*),\phi)$ in 
the convergence rate depends on the variance $\sigma(x^*)^2$ and on the distance of the $k_0(x^*)$ initial iterates to $x^*$, 
where $k_0(x^*)$ and $\phi$ are chosen such that \eqref{ee23} is satisfied. Under Assumption \ref{extra.cte.control}(ii), 
since $K_\epsilon$ does 
not depend on $x^*\in X^*$, we get indeed the uniform estimate:
\begin{equation}\label{yyy}
\sup_{\epsilon>0}\epsilon K_\epsilon\le\inf_{x^*\in X^*} \mathsf{Q}_\infty(x^*,k_0(x^*),\phi).
\end{equation}
In view of \eqref{yyy} , the performance of method \eqref{algorithm.extra.cte1}-\eqref{algorithm.extra.cte2} under non-uniform variance 
depends on the $x^*\in X^*$ such that $\mathsf{Q}_\infty(x^*,k_0(x^*),\phi)$ is minimal.

\quad

\begin{prop}[Rate and oracle complexity for $m=1$: non-uniform variance]\label{pp1} 
Consider Assumptions \ref{existence}-\ref{extra.cte.control}(i) for some $x^*\in X^*$ and
take $\alpha_k\equiv\alpha\in(0,1/\sqrt{6}L)$. Define ${\mathcal{N}}_k$ as
\begin{equation}\label{rem:conv:sample:rate}
\mathcal{N}_k=\left\lceil\theta(k+\mu)(\ln (k+\mu))^{1+b}\right\rceil
\end{equation}
for any $\theta>0$, $b>0$, $\epsilon>0$ and $2<\mu\leq\epsilon^{-1}$. Choose $\phi\in(0,\frac{\sqrt{5}-1}{2})$ and let $k_0(x^*)$ be the minimum natural number satisfying 
\begin{equation}\label{rem:conv:k0}
k_0(x^*)\ge\exp\left[\left(\frac{2cC_2^2\hat\alpha^2\sigma(x^*)^2}{\phi b\theta}\right)^{1/b}\right]-\mu+1.
\end{equation}

Then Theorem \ref{extra.cte.convergence} holds and there are non-negative constants $\mathsf{\overline Q}(x^*)$, $\mathsf{P}(x^*)$ and $\mathsf{I}(x^*)$ depending on $x^*$, $k_0(x^*)$ and $\phi$ such that for all $\epsilon>0$, there exists $K:=K_\epsilon\in\mathbb{N}$ such that
\begin{eqnarray}
\esp[r_\alpha(x^{K})^2]\le\epsilon \le 
\frac{\max\{1,\theta^{-2}\}\mathsf{\overline Q}(x^*)}{K},\label{prop:conv:nonunif:rate}\\ 
\sum_{k=1}^{K} 2\mathcal{N}_k \le  
\frac{\max\{1,\theta^{-4}\}\max\{1,\theta\}\mathsf{I}(x^*)\left\{\left[\ln\left(\mathsf{P}(x^*)\epsilon^{-1}\right)\right]^{1+b}+\frac{1}{\mu}\right\}}{\epsilon^2}.\label{prop:conv:nonunif:oracle}
\end{eqnarray}
\end{prop}
\begin{proof}
For $\phi\in(0,\frac{\sqrt{5}-1}{2})$, we want $k_0:=k_0(x^*)$ to satisfy \eqref{ee23} of Theorem \ref{thm:convergence:rate}. We have
\begin{eqnarray}
\sum_{k\ge k_0}\frac{1}{\mathcal{N}_k}&\le &\theta^{-1}\sum_{k\ge k_0}\frac{1}{(k+\mu)(\ln (k+\mu))^{1+b}}\nonumber\\
&\le &\theta^{-1}\int_{k_0-1}^\infty\frac{dt}{(t+\mu)(\ln (t+\mu))^{1+b}}\nonumber\\
&=&\frac{\theta^{-1}}{b(\ln (k_0-1+\mu))^{b}}.\label{prop:pp1:k0}
\end{eqnarray}
From \eqref{prop:pp1:k0} and \eqref{ee23}, it is enough to choose $k_0$ as the minimum natural number such that the 
RHS of \eqref{prop:pp1:k0} is less than $\phi/\mathsf{D}(x^*)$. Using the definition of $\mathsf{D}(x^*)$ 
in \eqref{def:D}, it is enough to choose $k_0$ as in \eqref{rem:conv:k0}.
	
We now give an estimate of $\mathsf{Q}_\infty(x^*,k_0,\phi)$ as defined in \eqref{def:Qx}. Set $\lambda:=2c\hat\alpha^2C_2^2$ with $c$ as defined in Remark \ref{estimate.C}. From the definitions \eqref{def:D} and \eqref{def:ak:bk}, we have the bound
\begin{eqnarray}
\mathsf{D}(x^*)\mathsf{a}^\infty_0+
\mathsf{D}(x^*)^2\mathsf{b}^\infty_0& \le &
\int_{-1}^\infty\frac{\lambda\theta^{-1}\sigma(x^*)^2dt}{(t+\mu)(\ln (t+\mu))^{1+b}}+\label{rem:bound:E}	\\
+\int_{-1}^\infty\frac{\lambda^2\theta^{-2}\sigma(x^*)^4dt}{(t+\mu)^2(\ln (t+\mu))^{2+2b}}
&\le &\frac{\lambda\theta^{-1}\sigma(x^*)^2}{b(\ln(\mu-1))^b}+\frac{\lambda^2\theta^{-2}\sigma(x^*)^4}{(\mu-1) (1+2b)[\ln(\mu-1)]^{1+2b}}.\nonumber
\end{eqnarray}
Using Theorem \ref{thm:convergence:rate}, \eqref{rem:bound:E}, the definition of $\mathsf{Q}_\infty(x^*,k_0,\phi)$ as defined by \eqref{def:ak:bk}-\eqref{def:Qx}, we prove \eqref{prop:conv:nonunif:rate}, noting 
that $\mathsf{\overline Q}(x^*)$ is specified as in Remark \ref{rem:constants:rate:non:uniform}.

We now prove \eqref{prop:conv:nonunif:oracle}. Denoting $\mathsf{Q}_\infty(x^*):=\mathsf{Q}_\infty(x^*,k_0,\phi)$ 
and using $K:=K_\epsilon\le \mathsf{Q}_\infty(x^*)/\epsilon$, $\mu\epsilon\le1$ 
and $\mathcal{N}_k\le\theta(k+\mu)(\ln(k+\mu))^{1+b}+1$, we have
\begin{eqnarray}
\sum_{k=1}^{K} 2\mathcal{N}_k &\le &
\max\{\theta,1\}\sum_{k=1}^{K}2\left[(k+\mu)(\ln(k+\mu))^{1+b}+1\right]\nonumber\\
&\le &\max\{\theta,1\}K(K+2\mu)\left[(\ln (K+\mu))^{1+b}+\frac{2}{K+2\mu}\right]\nonumber\\
&\le &\max\{\theta,1\}\frac{\left\{\left[\ln\left(\mathsf{Q}_\infty(x^*)\epsilon^{-1}+
\epsilon^{-1}\right)\right]^{1+b}+\mu^{-1}\right\}\mathsf{Q}_\infty(x^*)\left(\mathsf{Q}_\infty(x^*)+2\right)}{\epsilon^2}.\label{zzz}
\end{eqnarray}
We now use \eqref{zzz} with $\mathsf{Q}_\infty(x^*)(\mathsf{Q}_\infty(x^*)+2)\le(\mathsf{Q}_\infty(x^*)+2)^2$, 
the definition of $\mathsf{Q}_\infty(x^*,k_0,\phi)$ as in \eqref{def:ak:bk}-\eqref{def:Qx}, 
the bound \eqref{rem:bound:E} and the fact that $(a+b+c)^2\le3(a^2+b^2+c^2)$ in order to prove \eqref{prop:conv:nonunif:oracle}, 
where $\mathsf{I}(x^*)$ and $\mathsf{P}(x^*)$ are given in Remark \ref{rem:constants:rate:non:uniform}.
\end{proof}
\begin{rem}[Constants]\label{rem:constants:rate:non:uniform}
\emph{
We make use of the following definitions for the sake of clarity:
\begin{eqnarray}
\mathcal{A}_{\mu,b}:=\frac{2c\hat\alpha^2 C_2^2}{b[\ln(\mu-1)]^b},\quad
\mathcal{B}_{\mu,b}:=\frac{(2c\hat\alpha^2 C_2^2)^2}{(\mu-1) (1+2b)[\ln(\mu-1)]^{1+2b}},\label{def:mathAB}\\
\mathsf{\overline Q}(d,A,J):=2\rho^{-1}d^2+2\rho^{-1}A\left(1+J\right),\label{def:Q:overline}\\
\mathsf{I}(d,A,J):=12\rho^{-2}d^4+12\rho^{-2}A^2\left(1+J\right)^2+1,\label{def:I}
\end{eqnarray}
using the definition of $c$ in Remark \ref{estimate.C} and \eqref{def:ak:bk}. Using the definitions \eqref{def:D}, \eqref{def:ak:bk}-\eqref{def:Qx} and \eqref{def:mathAB}-\eqref{def:I}, 
the constants in the statement of Proposition \ref{pp1} 
are given by 
\begin{eqnarray*}
\mathsf{\overline Q}(x^*)&:=&\mathsf{\overline Q}(\Vert x^0-x^*\Vert,\sigma(x^*)^2\mathcal{A}_{\mu,b}+
\sigma(x^*)^4\mathcal{B}_{\mu,b},\mathsf{J}(x^*,k_0(x^*),\phi)),\\
\mathsf{P}(x^*)&:=&\mathsf{Q}_\infty(x^*,k_0(x^*),\phi)+1,\\
\mathsf{I}(x^*)&:=&\mathsf{I}(\Vert x^0-x^*\Vert,\sigma(x^*)^2\mathcal{A}_{\mu,b}+\sigma(x^*)^4\mathcal{B}_{\mu,b},\mathsf{J}(x^*,k_0(x^*),\phi)).
\end{eqnarray*}
Given $\epsilon>0$, we may use the definitions of $\mathsf{\overline{Q}}(x^*)$, $\mathsf{I}(x^*)$ and $\mathsf{P}(x^*)$ and optimize the estimates given in \eqref{prop:conv:nonunif:rate}-\eqref{prop:conv:nonunif:oracle} over $(\hat\alpha,\theta)$, obtaining optimal constants in terms of $L$ and $\sigma(x^*)^2$. For simplicity we do not carry this procedure here.
}
\end{rem}

\quad

We give next sharper estimates in the case the variance is \emph{uniform} over $X^*$ or $X$. We state them without proofs 
since they follow the same proof line of Theorem \ref{thm:convergence:rate} and Proposition \ref{pp1}, 
but using Proposition \ref{prop:telescopeandmartingale}(ii) and Proposition \ref{prop:lp:bound}(ii), when 
the variance is uniform over $X^*$, and Proposition \ref{prop:telescopeandmartingale}(iii), when the variance is uniform over $X$. 
Define:
\begin{eqnarray}
\mathsf{D}_\sigma:=2c\hat\alpha^2 C_2^2\sigma^2,\label{def:D:unif:x*}\\
\mathsf{J}(\ell,\phi):=\frac{1+\max_{0\le k\le \ell}\esp[\dist(x^k,X^*)^2]}{1-\phi-\phi^2},\label{def:J:unif:x*}\\
\mathsf{Q}_k(\sigma,\ell,\phi):=2\rho^{-1}\left\{\dist(x^0,X^*)^2+\left(1+\mathsf{J}(\ell,\phi)\right)
\left(\mathsf{D}_\sigma\mathsf{a}_0^k+\mathsf{D}_\sigma^2\mathsf{b}_0^k\right)\right\},\label{def:Q:unif:x*}\\
\mathsf{\widetilde Q}_k(\sigma):=2\rho^{-1}\left\{\dist(x^0,X^*)^2+17C_2^2\hat\alpha^2\sigma^2
\sum_{i=0}^k\frac{1}{N_{i,\min}}\right\},\label{def:Q:unif}
\end{eqnarray}
using $c$ as defined in Remark \ref{estimate.C} and the definitions in \eqref{def:ak:bk}. 

\quad

\begin{thm}[Convergence rate: uniform variance]
\label{thm:convergence:rate:unif}
Consider Assumptions \ref{existence}-\ref{extra.cte.control} and take $\alpha_k\equiv\alpha\in(0,1/\sqrt{6}L)$.	

Suppose first that Assumption \ref{extra.cte.control}(ii) holds and $\sup_{x^*\in X^*}\sigma(x^*)\le\sigma$ 
for some $\sigma>0$. Take $\phi\in(0,\frac{\sqrt{5}-1}{2})$ and $k_0:=k_0(\sigma)\in\mathbb{N}$ such that 
$$
\sum_{k\ge k_0}\frac{1}{\mathcal{N}_k}\le\frac{\phi}{\mathsf{D}_\sigma}.
$$ 
Then, for all $\epsilon>0$, there exists $K_\epsilon\in\mathbb{N}$, satisfying
$$
\esp[r_\alpha\left(x^{K_\epsilon}\right)^2]\le\epsilon
\le \frac{\mathsf{Q}_\infty(\sigma,k_0(\sigma),\phi)}{K_\epsilon}.
$$

Suppose now that Assumption \ref{extra.cte.control}(iii) holds for some $\sigma>0$. 
Then, for all $\epsilon>0$, there exists $K_\epsilon\in\mathbb{N}$, satisfying
$$
\esp[r_\alpha\left(x^{K_\epsilon}\right)^2]\le\epsilon
\le\frac{\mathsf{\widetilde Q}_\infty(\sigma)}{K_\epsilon}.
$$
\end{thm}

\begin{prop}[Rate and oracle complexity for $m=1$: uniform variance]\label{pp3}
Consider Assumptions \ref{existence}-\ref{extra.cte.control}, take $\alpha_k\equiv\alpha\in(0,1/\sqrt{6}L)$ and suppose $\sup_{x^*\in X^*}\sigma(x^*)\le\sigma$ for some $\sigma>0$. Define ${\mathcal{N}}_k$ as
$$
\mathcal{N}_k=\left\lceil\theta(k+\mu)(\ln (k+\mu))^{1+b}\right\rceil
$$
for any $\theta>0$, $b>0$, $\epsilon>0$ and $2<\mu\leq\epsilon^{-1}$. Then the following holds:
\begin{itemize}
\item[(i)] Suppose Assumption \ref{extra.cte.control}(ii) holds. Choose $\phi\in(0,\frac{\sqrt{5}-1}{2})$ 
and $k_0:=k_0(\sigma)\in\mathbb{N}$ such that
$$
k_0\ge\exp\left[\left(\frac{2cC_2^2\hat\alpha^2\sigma^2}{\phi b\theta}\right)^{1/b}\right]-\mu+1.
$$
Then Theorem \ref{extra.cte.convergence} holds and there exist 
non-negative constants $\mathsf{\overline Q}(\sigma)$, $\mathsf{P}(\sigma)$ and $\mathsf{I}(\sigma)$ depending on $\sigma$, $k_0(\sigma)$ and $\phi$ such that for all $\epsilon>0$, there exists $K:=K_\epsilon\in\mathbb{N}$ such that
\begin{eqnarray*}
\esp[r_\alpha(x^{K})^2]\le\epsilon \le 
\frac{\max\{1,\theta^{-2}\}\mathsf{\overline Q}(\sigma)}{K},\\ 
\sum_{k=1}^{K} 2\mathcal{N}_k \le  
\frac{\max\{1,\theta^{-4}\}\max\{1,\theta\}\mathsf{I}(\sigma)\left\{\left[\ln\left(\mathsf{P}(\sigma)\epsilon^{-1}\right)\right]^{1+b}
+\frac{1}{\mu}\right\}}{\epsilon^2}.
\end{eqnarray*}
\item[(ii)] Suppose that Assumption \ref{extra.cte.control}(iii) holds. Then Theorem \ref{extra.cte.convergence} 
holds and there exist non-negative constants $\mathsf{\widetilde Q}(\sigma)$, 
$\mathsf{\widetilde P}(\sigma)$ and $\mathsf{\widetilde I}(\sigma)$ depending on $\sigma$ such 
that for all $\epsilon>0$, there exists $K:=K_\epsilon\in\mathbb{N}$ such that
\begin{eqnarray*}
\esp[r_\alpha(x^{K})^2]\le\epsilon\le\frac{\max\{1,\theta^{-1}\}\mathsf{\widetilde Q}(\sigma)}{K_\epsilon},\\
\sum_{k=1}^{K}2\mathcal{N}_k\le
\frac{\max\{1,\theta^{-2}\}\max\{1,\theta\}\mathsf{\widetilde I}(\sigma)
\left\{\left[\ln\left(\mathsf{\widetilde P}(\sigma)\epsilon^{-1}\right)\right]^{1+b}+\frac{1}{\mu}\right\}}{\epsilon^2}.
\end{eqnarray*}
\end{itemize}
\end{prop}
\begin{rem}[Constants]\label{rem:constants:rate:uniform}
\emph{
We recall the definitions \eqref{def:ak:bk}, \eqref{def:mathAB}-\eqref{def:I} and \eqref{def:D:unif:x*}-\eqref{def:Q:unif}. The constants in the statement of Proposition \ref{pp3}(i) are given 
by $\mathsf{\overline Q}(\sigma):=\mathsf{\overline Q}(\dist(x^0,X^*),\sigma^2\mathcal{A}_{\mu,b}+\sigma^4\mathcal{B}_{\mu,b},\mathsf{J}(k_0(\sigma),\phi))$, 
$\mathsf{P}(\sigma):=\mathsf{Q}_\infty(\sigma,k_0(\sigma),\phi)+1$ 
and $\mathsf{I}(\sigma)$ $:=\mathsf{I}(\dist(x^0,X^*),\sigma^2\mathcal{A}_{\mu,b}+\sigma^4\mathcal{B}_{\mu,b},\mathsf{J}(k_0(\sigma),\phi))$. 
For item (ii) they are given by 
\begin{eqnarray*}
\mathsf{\widetilde Q}(\sigma)&:=&2\rho^{-1}\dist(x^0,X^*)^2+2\rho^{-1}\frac{17C_2^2\hat\alpha^2\sigma^2}{b(\ln(\mu-1))^{b}},\label{def:Q:unif:infty}\\
\mathsf{\widetilde I}(\sigma)&:=&12\rho^{-2}\dist(x^0,X^*)^4+12\rho^{-2}\frac{17^2C_2^4\hat\alpha^4\sigma^4}{b^2(\ln(\mu-1))^{2b}}+1,\label{def:I:unif}
\end{eqnarray*}
and $\mathsf{\widetilde{P}}(\sigma):=\mathsf{\widetilde Q}_\infty(\sigma)+1$.
}
\end{rem}

\quad

We now turn our attention to the distributed solution of a Cartesian SVI for a large network ($m\gg1$). 
If a \emph{decentralized sampling} is used, then higher order factors of $m$ appear in rate and complexity. 
The next results shows that if, in addition, a deterministic and decreasing sequence of exponents 
$\{b_i\}_{i=1}^m$ and an approximate estimate of the network dimension $m$ is coordinated, then 
the oracle complexity is proportional to $m$ (up to a scaling in the sampling rate). 

\quad

\begin{prop}[Rate and oracle complexity for a network]\label{pp2}
Consider Assumptions \ref{existence}-\ref{extra.cte.control}(i) for some $x^*\in X^*$ and
take $\alpha_k\equiv\alpha\in(0,1/\sqrt{6}L)$. Under Assumptions \ref{existence}-\ref{extra.cte.control}(i) with Assumption \ref{extra.cte.unbias}(i) (centralized sampling), the results of Proposition \ref{pp1} hold. 

Consider now Assumption \ref{extra.cte.unbias}(ii) (decentralized sampling). Let $\{b_i\}_{i=1}^m$ be a positive sequence such that
\begin{eqnarray}
&& N_{k,i}=\left\lceil\theta_i(k+\mu_i)^{1+a}(\ln (k+\mu_i))^{1+b_i}\right\rceil,\label{rem:conv:sample:rate:m}\\
&& b_1\ge b_i+2\ln(i+1)-\ln\mathsf{S},
\label{rem:exponents}
\end{eqnarray}
for any $\theta_i>0$, $a>0$, $\mathsf{S}\ge1$, $\epsilon>0$, $2<\mu_i\le\epsilon^{-1}$. Choose $\phi\in(0,\frac{\sqrt{5}-1}{2})$ and let $k_0(x^*)$ be the minimum natural number greater than $e-\mu_{\min}+1$ such that
\begin{equation}
k_0(x^*)\ge\left[\frac{2cC_2^2\hat\alpha^2\sigma(x^*)^2 }{\phi\theta_{\min} b_{\min}}\right]^{1/a}-\mu_{\min}+1.\label{rem:conv:k0:m}
\end{equation}

Then Theorem \ref{extra.cte.convergence} holds and there exist 
non-negative constants $\mathsf{\widehat Q}(x^*)$, $\mathsf{\widehat P}(x^*)$ and $\mathsf{\widehat I}(x^*)$ 
depending on $x^*$, $k_0(x^*)$ and $\phi$ such that for all $\epsilon>0$, there exists $K:=K_\epsilon\in\mathbb{N}$ such that
\begin{eqnarray}
\esp[r_\alpha(x^{K})^2]\le\epsilon \le 
\frac{\mathsf{\widehat Q}(x^*)}{K},&&\label{rem:conv:rate:m}\\ 
\sum_{k=1}^{K} \sum_{i=1}^m2N_{k,i} \le  
\frac{\mathsf{S}\max\{\theta_{\max},1\}\left\{\ln\left(\mathsf{\widehat P}(x^*)\epsilon^{-1}\right)\right\}^{1+b_1}\mathsf{\widehat{I}}(x^*)}{\epsilon^{2+a}}.
\end{eqnarray}
(Above, the subscripts ``$\min$'' and ``$\max$'' refer, respectively, to the minimal and maximal terms of the corresponding sequences).
\end{prop}
\begin{proof}
In the sequel we will  use the following estimate. For any $k\in\mathbb{N}_0$, $a>0$, $0<b<1$, $\mu>1$, 
\begin{equation}
\int_k^\infty\frac{dt}{(t+\mu)^{1+a}(\ln(t+\mu))^{1+b}}\le\max\left\{\frac{1}{a(k+\mu)^a},
\frac{1}{(k+\mu)^ab[\ln(k+\mu)]^b}\right\}.\label{rem:estimate:m}
\end{equation}

For $\phi\in(0,\frac{\sqrt{5}-1}{2})$ we want $k_0:=k_0(x^*)$ to satisfy \eqref{ee23} of Theorem \ref{thm:convergence:rate}. 
Since $\mathcal{N}_k$ is the harmonic average of $\{N_{k,i}\}_{i=1}^m$ and 
$N_{k,i}\ge\theta_{\min}(k+\mu_{\min})^{1+a}[\ln (k+\mu_{\min})]^{1+b_{\min}}$ for all $i\in[m]$, we get from \eqref{rem:estimate:m}:
$$
\sum_{k\ge k_0}\frac{1}{\mathcal{N}_{k}}\le 
\theta_{\min}^{-1}\sum_{k\ge k_0}\frac{1}{(k+\mu_{\min})^{1+a}[\ln (k+\mu_{\min})]^{1+b_{\min}}}
$$
\begin{equation}\label{cccc}	
\le\frac{\theta_{\min}^{-1}}{(k_0-1+\mu_{\min})^ab_{\min}[\ln(k_0-1+\mu_{\min})^{b_{\min}}]}\le
\frac{\theta_{\min}^{-1}}{(k_0-1+\mu_{\min})^ab_{\min}},
\end{equation}		
if $k_0\ge e-\mu_{\min}+1$. From \eqref{cccc} and \eqref{ee23}, it is enough to choose $k_0$ as the minimum natural number 
greater than $e-\mu_{\min}+1$ such that the RHS of \eqref{cccc} is less than $\phi/\mathsf{D}(x^*)$. Using 
the definition of $\mathsf{D}(x^*)$ in \eqref{def:D}, it is enough to choose $k_0$ as in \eqref{rem:conv:k0:m}.

Next we estimate the value of $\mathsf{Q}_\infty(x^*,k_0,\phi)$ as defined in \eqref{def:Qx}. Recall $\frac{1}{\mathcal{N}_k}=\sum_{i=1}^m\frac{1}{N_{k,i}}$ and set $\lambda:=2c\hat\alpha^2C_2^2$ with $c$ as defined in Remark \ref{estimate.C}. Definitions \eqref{def:D} and \eqref{def:ak:bk} imply
\begin{eqnarray}	
\mathsf{D}(x^*)\mathsf{a}^\infty_0+
\mathsf{D}(x^*)^2\mathsf{b}^\infty_0&\le &
\sigma(x^*)^2\sum_{k\ge0}\sum_{i=1}^m\frac{\lambda\theta_i^{-1}}{(k+\mu_i)^{1+a}(\ln (k+\mu_i))^{1+b_i}}+\nonumber\\
&&\sigma(x^*)^4\sum_{k\ge0}\left[\sum_{i=1}^m\frac{\lambda\theta_i^{-1}}{(k+\mu_i)^{1+a}(\ln (k+\mu_i))^{1+b_i}}\right]^2\label{rem:bound:E:m:eq1}.
\end{eqnarray}
The first summation in \eqref{rem:bound:E:m:eq1} is bounded by
\begin{equation}\label{rem:bound:E:m:eq2}
\sum_{i=1}^m\sum_{k\ge0}\frac{\lambda\theta_i^{-1}}{(k+\mu_i)^{1+a}}\le\sum_{i=1}^m\int_{-1}^\infty\frac{\lambda\theta_i^{-1}dt}{(t+\mu_i)^{1+a}}\le
\sum_{i=1}^m\frac{\lambda}{\theta_i a(\mu_i-1)^a}=:\mathcal{A}_m.
\end{equation}
Using estimate \eqref{rem:estimate:m}, the second summation in \eqref{rem:bound:E:m:eq1} is bounded by
\begin{eqnarray}
&&\sum_{i=1}^m\sum_{j=1}^m\sum_{k\ge0}\frac{\lambda^2\theta_i^{-1}\theta_j^{-1}}{(k+\mu_{\min})^{2+2a}[\ln(k+\mu_{\min})]^{2+b_i+b_j}}\nonumber\\
&&\le\frac{1}{\vartheta}\left\{\sum_{i=1}^m\frac{\lambda}{\theta_i[\ln(\mu_{\min}-1)]^{b_i}}\right\}^2=:\mathcal{B}_m,\label{rem:bound:E:m:eq3}
\end{eqnarray}
where $\vartheta:=(1+2b_{\min})(\mu_{\min}-1)^{1+2a}\ln(\mu_{\min}-1)$. From Theorem \ref{thm:convergence:rate}, \eqref{rem:bound:E:m:eq1}-\eqref{rem:bound:E:m:eq3} and  definition of $\mathsf{Q}_\infty(x^*,k_0,\phi)$ as specified in \eqref{def:ak:bk}-\eqref{def:Qx}, we prove \eqref{rem:conv:rate:m}, noting that $\mathsf{\widehat Q}(x^*)$ is specified as in Remark \ref{rem:constants:rate:network}.

We now prove the bound on the oracle complexity. Let $\mathsf{\mathsf{Q}}_\infty(x^*):=\mathsf{\mathsf{Q}}_\infty(x^*,k_0,\phi)$.
Using the facts that $K\le \mathsf{Q}_\infty(x^*)/\epsilon$, $\mu_i\leq\epsilon^{-1}$ and the definition of $N_{k,i}$, we have
\begin{eqnarray}	
\sum_{k=1}^{K}\sum_{i=1}^m 2N_{k,i}&\le &\sum_{k=1}^{K}\sum_{i=1}^m 2\left[\theta_i(k+\mu_i)^{1+a}(\ln(k+\mu_i))^{1+b_i}+1\right]\nonumber\\
&\le &2\max\{\theta_{\max},1\}K\sum_{i=1}^m	
\left[(K+\mu_i)^{1+a}\left(\ln \left(K+\mu_i\right)\right)^{1+b_i}+1\right]\nonumber\\
&\le &4\max\{\theta_{\max},1\}K\sum_{i=1}^m	
\left[(K+\mu_i)^{1+a}\left(\ln \left(K+\mu_i\right)\right)^{1+b_i}\right]\nonumber\\
&\le &4\Phi\frac{\left(\mathsf{Q}_\infty(x^*)+1\right)^{2+a}}{\epsilon^{2+a}}
\sum_{i=1}^m\left(\ln\left(\mathsf{Q}_\infty(x^*)\epsilon^{-1}+\epsilon^{-1}\right)\right)^{1+b_i},\label{rem:complex:m:aux}
\end{eqnarray}
using the fact that $1\le (K+\mu_i)^{1+a}\left(\ln \left(K+\mu_i\right)\right)^{1+b_i}$ for $i\in[m]$ 
in the  third inequality, and defining  
$\Phi:=\max\{\theta_{\max},1\}$ 
in the last inequality.

Set $h:=\ln\left(\mathsf{Q}_\infty(x^*)\epsilon^{-1}+\epsilon^{-1}\right)$ with $h\ge e$ for sufficiently small 
$\epsilon>0$. By definition of $\{b_i\}_{i=1}^m$ we have, for $i\in[m]$,
\begin{equation}\label{pppp}
b_1\ge b_i+2\ln(i+1)-\ln\mathsf{S}\ge b_i+\frac{2\ln(i+1)-\ln\mathsf{S}}{\ln h}\Rightarrow h^{b_i}\le\frac{\mathsf{S} h^{b_1}}{(i+1)^2}.
\end{equation} 
From \eqref{pppp}  we obtain that 
\begin{equation}\label{rem:bound:E:m:eq4}
\sum_{i=1}^mh^{b_i}\le \mathsf{S} h^{b_1}\sum_{i=1}^{m}\frac{1}{(i+1)^2}\le \mathsf{S} h^{b_1}.
\end{equation}
From the bounds \eqref{rem:bound:E:m:eq1}-\eqref{rem:bound:E:m:eq4},  the definitions of $h$ and $\mathsf{Q}_\infty(x^*,k_0,\phi)$, 
as specified in \eqref{def:ak:bk}-\eqref{def:Qx} and the fact that $(x+y+z)^{2+a}\le3^{1+a}(x^{2+a}+y^{2+a}+z^{2+a})$, 
we obtain the required bound on $\sum_{k=1}^K\sum_{i=1}^m2N_{k,i}$, noting that $\mathsf{\widehat I}(x^*)$ 
and $\mathsf{\widehat{P}}(x^*)$ are specified as in Remark \ref{rem:constants:rate:network}.
\end{proof}
\begin{rem}[Constants]\label{rem:constants:rate:network}
\emph{
Define:
\begin{eqnarray}\label{def:I:hat}
\mathsf{\widehat{I}}(d,A,J,\nu):=4\cdot3^{\nu-1}\left\{(2\rho^{-1})^{\nu}d^{2\nu}+(2\rho^{-1})^{\nu}A^\nu\left[1+J\right]^\nu+1\right\},
\end{eqnarray}
using \eqref{def:ak:bk}. In view of \eqref{def:D}, \eqref{def:ak:bk}-\eqref{def:Qx}, 
\eqref{def:Q:overline}, \eqref{def:I:hat} and \eqref{rem:bound:E:m:eq2}-\eqref{rem:bound:E:m:eq3}, 
the constants in the statement of Proposition \ref{pp2} are given by 
\begin{eqnarray*}
\mathsf{\widehat Q}(x^*)&:=& \mathsf{\overline Q}(\Vert x^0-x^*\Vert,\sigma(x^*)^2\mathcal{A}_m+\sigma(x^*)^4\mathcal{B}_m,\mathsf{J}(x^*,k_0(x^*),\phi)),\\
\mathsf{\widehat P}(x^*)&:=&\mathsf{Q}_\infty(x^*,k_0(x^*),\phi)+1,\\
\mathsf{\widehat I}(x^*)&:=&\mathsf{\widehat I}(\Vert x^0-x^*\Vert,\sigma(x^*)^2\mathcal{A}_{m}+\sigma(x^*)^4\mathcal{B}_{m},\mathsf{J}(x^*,k_0(x^*),\phi),2+a).
\end{eqnarray*}
}
\end{rem}
\begin{rem}[Oracle complexity of $\mathcal{O}(m)$]
\emph{
For the choice of parameters \eqref{rem:conv:sample:rate:m}-\eqref{rem:exponents}, if we have $\theta_{i}\sim\theta m$ 
for some $\theta>0$ 
then $\mathcal{A}_m\lesssim\frac{\theta^{-1}}{a(\mu_{\min}-1)^a}$ and $\mathcal{B}_m\lesssim\frac{\theta^{-2}}{(\mu_{\min}-1)^{1+2a}}$, 
where $\mathcal{A}_m$ and $\mathcal{B}_m$ are defined in \eqref{rem:bound:E:m:eq2}-\eqref{rem:bound:E:m:eq3}. Also, 
$b_{\min}\le b_1+\ln\mathsf{S}-2\ln(m+1)$ so that it is enough to choose $b_1>2\ln(m+1)-\ln\mathsf{S}$, which is reasonably 
small in terms of $m$. Finally, the bound on the oracle complexity in Proposition \ref{pp2} is of order 
$\max\{1,\theta^{-2(2+a)}\}\theta_{\max}\lesssim\max\{\theta,\theta^{-(3+2a)}\}m$. Moreover, the sampling is robust in the sense that the 
convergence rate is proportional to $\max\{1,\theta^{-2}\}$ and the oracle complexity is proportional to $\max\{\theta,\theta^{-(3+2a)}\}$. We remark that improvements can be achieved if a coordination $\mu_{\min}\sim\epsilon^{-1}$ is possible (given a prescribed tolerance 
$\epsilon>0$). For simplicity we do not present the analogous results of Proposition \ref{pp3} for the case $m\gg1$ under 
Assumption \ref{extra.cte.control}(iii). In that case, the estimates depend on $\dist(x^0,X^*)$, the rate is proportional 
to $\max\{1,\theta^{-1}\}$ and the oracle complexity is proportional to $\max\{\theta,\theta^{-(2+a)}\}m$.
}
\end{rem}
%

\subsubsection{Comparison of complexity estimates}\label{sss3.4.1}
Next, we briefly compare our complexity results in terms of the quadratic natural residual, given in  this section, with 
related results presented in previous work in terms of other merit functions for the stochastic variational inequality. 
As commented in Subsection \ref{ss1.1}, the quadratic natural residual and the D-gap function are equivalent merit functions. 
An immediate result is that the previous complexity analysis,  given in Theorems \ref{thm:convergence:rate}-\ref{thm:convergence:rate:unif} 
and Propositions \ref{pp1}-\ref{pp2} in terms of the quadratic natural residual, are also valid in terms of the D-gap function. 
In this sense, our rate of convergence of $\mathcal{O}(1/K)$ in terms of the D-gap function improves the rate $\mathcal{O}(1/\sqrt{K})$ in terms of 
the dual gap-functions analyzed in \cite{nem, lan, uday2, uday5}.

By Proposition \ref{pp1} and Remark \ref{rem:constants:rate:non:uniform}, if Assumption \ref{extra.cte.control}(ii) holds, then the algorithm performance, in terms of the convergence rate and oracle complexity, depends on some $x^*\in X^*$ such that
$
\sigma(x^*)^4\cdot\max_{0\le k\le k_0(x^*)}\esp[\Vert x^k-x^*\Vert^2]
$	
is minimal, that is to say, we have a trade-off between variance of the oracle error and distance to initial iterates. 
We also remark that the sampling rate $\mathcal{N}_k$ possesses a \emph{robust property}: a scaling in the sampling rate by  
a factor $\theta$, keeps the algorithm running with a proportional scaling of $\max\{1,\theta^{-2}\}$ in the rate 
and $\max\{\theta,\theta^{-3}\}$ in the oracle complexity (see \cite{shapiro2} for discussion on robust algorithms). 
By Proposition \ref{pp3}, when the variance is bounded by $\sigma^2$ over $X^*$, the estimates depend on 
$
\sigma^4\max_{0\le k\le k_0(\sigma)}\esp[\dist(x^k,X^*)^2]
$ and $k_0(\sigma)$ is independent of any $x^*\in X^*$. When the variance is uniform over $X$, the estimates depend 
only on $\dist(x^0,X^*)$ and a scaling factor $\theta$ in the sampling rate implies a factor of $\max\{1,\theta^{-1}\}$ 
in the rate and of $\max\{\theta,\theta^{-1}\}$ in the oracle complexity. In the estimates of Propositions \ref{pp1}-\ref{pp3}, 
we may obtain optimal constants in terms of $L$, the variance and distance to initial iterates by optimizing over $(\hat\alpha,\theta)$. 
Interestingly, in the case of a compact feasible set, the estimates do not depend on $\diam(X)$, as in \cite{nem,lan}, but rather on the 
distance of the initial iterates to $X^*$, which is a sharper result. In the case of networks the same conclusions hold, 
except that the dependence in the dimension is higher if a decentralized sampling is used. From Proposition \ref{pp2}, 
if a distributed sampling is used and a coordination of a rapid decreasing sequence of positive numbers is offered (in any order), 
then the oracle complexity depends linearly on the size of the network (up to a scaling factor in the sampling rate).

We briefly compare our convergence rate and complexity bounds presented in Propositions \ref{pp1} and \ref{pp3} with those 
in \cite{lan} (Corollaries 3.2 and 3.4). In \cite{lan}, for a compact $X$ with uniform variance over $X$, 
the convergence rate obtained in terms of the dual gap function is of order
$
L\diam(X)^2K^{-1}+\sigma\diam(X)K^{-1/2},
$
and the oracle complexity is of order
$
L\diam(X)^2\epsilon^{-1}+\sigma^2\diam(X)^2\epsilon^{-2}.
$ 
For an unbounded $X$ with uniform variance over $X$, the convergence rate in terms of the relaxed dual-gap 
function $\tilde G(x,v)$ described in Subsection \ref{ss1.1} is of order 
$
L\Vert x^0-x^*\Vert^2K^{-1}+\sigma\Vert x^0-x^*\Vert^2K^{-1/2},
$
while the oracle complexity is of order
$
L\Vert x^0-x^*\Vert^2\epsilon^{-1}+\sigma^2\Vert x^0-x^*\Vert^4\epsilon^{-2}.
$
Note that the optimal constants in terms of $L$ and $\sigma$ in these bounds require tuning the stepsize to $L$ and $\sigma$. In the estimates given in Propositions \ref{pp1}-\ref{pp3}, the ``coercivity'' modulus $\rho^{-1}$ introduced 
by the extragradient step behaves qualitatively as $L$. 
We improve on the rate of convergence to $\mathcal{O}(1/K)$ with respect to the stochastic term 
$\mathcal{O}(\sigma/\sqrt{K})$ by reducing iteratively the variance.
Differently from \cite{lan}, our analysis is the same for a compact or unbounded $X$, in the sense that the same merit function 
is used. For the case of a compact $X$, our bounds depend on $\dist(x^0,X^*)$ rather $\diam(X)$ as in \cite{lan}, 
which is a sharper result. When the variance is uniform over an unbounded $X$, our bounds depend on $\dist(x^0,X^*)$ 
instead of $\Vert x^0-x^*\Vert$ for a given $x^*\in X^*$ as in \cite{lan}, which is also a sharper bound. 
We analyze the new case of non-uniform variance, which has a similar performance, except that the estimates depend on a point $x^*\in X^*$ with a minimum trade-off between variance and distances to a few initial iterates. 
Moreover, we include asymptotic convergence, which it is not reported in \cite{lan}.
	
\section{Concluding remarks}\label{s4}

In this work we propose an extragradient method for pseudo-monotone stochastic variational inequalities that 
combines the stochastic approximation methodology alongside an iterative variance reduction procedure. 
We obtain asymptotic convergence, non-asymptotic convergence rates and oracle complexity estimates and prove 
that the generated sequence is uniformly bounded in $L^p$. In order to achieve these properties, we require the operator to 
be just pseudo-monotone and Lipschitz-continuous. Our results give an accelerate rate with optimal oracle complexity, 
coping with unbounded feasible sets and an oracle with non-uniform variance. The method admits a robust sampling rate. 
We also include the analysis for the distributed solution of Cartesian SVIs.
	
A potential direction of future research is to obtain sharp complexity estimates for \emph{exponential convergence} 
of method \eqref{algorithm.extra.cte1}-\eqref{algorithm.extra.cte2}. In previous works \cite{nem,lan, friedlander}, exponential convergence 
is proved, assuming an \emph{uniform} tail bound for the oracle error, that is, that there exists $\sigma>0$ 
such that $\esp\left[\exp\left\{\frac{\Vert F(\xi,x)-T(x)\Vert^2}{\sigma^2}\right\}\right]\le\exp\{1\}$, for all $x\in X$. 
This assumption is not satisfied in general for unbounded feasible sets and, even for compact ones, $\sigma^2$ may be a conservative upper bound of the oracle variance at points of the trajectory of the method. 
Moreover, based on Section \ref{sss3.4.1}, in the case of a compact feasible set or uniform tail bound, we wish to study 
sharp complexity estimates with respect to the distance of the initial iterate to the solution set. We intend to make  
this analysis in a second paper assuming a non uniform tail bound in the spirit of Assumption \ref{extra.cte.control}(i)-(ii). 
Motivated by this work, another interesting line of research we intend to pursue is to verify if (extra)gradient methods 
with \emph{robust stepsizes} can achieve accelerated convergence rates with respect to the stochastic error.

Finally, we discuss error bounds on the solution set. It is well known that important classes of variational inequalities 
admit the natural residual as an error bound for the solution set, i.e., for all $\alpha>0,$ there exists $\delta>0$ such that 
$\dist(x,X^*)\lesssim r_\alpha(x)$ 
for all $x\in\re^n$ with $r_\alpha(x)\le\delta$.
This property holds, for example, for (i) semi-stable VIs, (ii) composite strongly monotone VIs such that $X$ is a polyhedron, (iii) VIs such that $T$ is linear and $X$ is a cone (see \cite{facchinei}). 
Item (ii) includes affine VIs and strongly monotone VIs. Item (iii) includes linear homogeneous complementarity 
problems and linear system of equations. When such property holds, the results of 
Theorems \ref{thm:convergence:rate}-\ref{thm:convergence:rate:unif} and Propositions \ref{pp1}-\ref{pp2} 
provide other classes of SVI's for which convergence of $\mathcal{O}(1/K)$ holds in terms of the mean-squared distance to the solution set. 
In the previous literature, such property was shown only for strongly pseudo-monotone or weak-sharp SVIs on a compact set. 
In a upcoming paper,
we intend to refine the complexity analysis for the case in which this error bound on the solution set is valid. 

\section*{Appendix. Proof of lemmas}
\quad

Proof of Lemma \ref{extra.cte.korpelevich}:
\begin{proof}
Let $x^*\in X^*$. In order to simplify the notation,
in the sequel we call $\widehat{F}(\epsilon^k_2,z^k):=T(z^k)+\epsilon^k_2$ and 
$y^k:=x^k-\alpha_k\widehat{F}(\epsilon^k_2,z^k)$, so that, $x^{k+1}=\Pi(y^k)$. For every $x\in X$, we have
\begin{eqnarray}
\Vert x^{k+1}-x\Vert^2 & = & \Vert \Pi(y^k)-x\Vert^2\nonumber\\
&\le & \Vert y^k-x\Vert^2-\Vert y^k-\Pi(y^k)\Vert^2\nonumber\\
& = & \Vert (x^k-x)-\alpha_k \widehat{F}(\epsilon^k_2,z^k)\Vert^2-
\Vert (x^k-x^{k+1})-\alpha_k \widehat{F}(\epsilon^k_2,z^k)\Vert^2\nonumber\\
& = & \Vert x^k-x\Vert^2-\Vert x^k-x^{k+1}\Vert^2+2\langle x-x^{k+1},\alpha_k\widehat{F}(\epsilon^k_2,z^k)\rangle\nonumber\\
& = & \Vert x^k-x\Vert^2-\Vert x^k-x^{k+1}\Vert^2+2\langle x-z^{k},\alpha_k\widehat{F}(\epsilon^k_2,z^k)\rangle+\nonumber\\
&& 2\langle z^k-x^{k+1},\alpha_k\widehat{F}(\epsilon^k_2,z^k)\rangle 
\label{extra.cte.eq3}\\
& = & \Vert x^k-x\Vert^2-\Vert (x^k-z^k)+(z^k-x^{k+1})\Vert^2+\nonumber\\
&& 2\langle z^k-x^{k+1},\alpha_k\widehat{F}(\epsilon^k_2,z^k)\rangle+
2\langle x-z^{k},\alpha_k\widehat{F}(\epsilon^k_2,z^k)\rangle\nonumber\\
& = & \Vert x^k-x\Vert^2-\Vert x^k-z^k\Vert^2-\Vert z^k-x^{k+1}\Vert^2\nonumber\\
&& -2\langle x^k-z^k,z^k-x^{k+1}\rangle+2\langle z^k-x^{k+1},\alpha_k\widehat{F}(\epsilon^k_2,z^k)\rangle+\nonumber\\
&& 2\langle x-z^{k},\alpha_k\widehat{F}(\epsilon^k_2,z^k)\rangle\nonumber\\
& = & \Vert x^k-x\Vert^2-\Vert x^k-z^k\Vert^2-\Vert z^k-x^{k+1}\Vert^2+\nonumber\\
&& 2\langle x^{k+1}-z^k,x^k-\alpha_k\widehat{F}(\epsilon^k_2,z^k)-z^k\rangle+
2\langle x-z^{k},\alpha_k\widehat{F}(\epsilon^k_2,z^k)\rangle,\nonumber
\end{eqnarray}	
using Lemma \ref{proj}(ii) in the inequality and simple algebra in the equalities. 

Looking at the fourth term $\mathsf{I}:=2\langle x^{k+1}-z^k,x^k-\alpha_k\widehat{F}(\epsilon^k_2,z^k)-z^k\rangle$ 
in the RHS of the last equality of
\eqref{extra.cte.eq3},
we take into account \eqref{algorithm.extra.cte1.F} and the fact that $\widehat{F}(\epsilon^k_2,z^k)=T(z^k)+\epsilon^k_2$, and apply
Lemma \ref{proj}(i) with $C=X$, $x=x^k-\alpha_k(T(x^k)+\epsilon^k_1)$ and $y=x^{k+1}\in X$, obtaining:
\begin{eqnarray}
\mathsf{I} &=& 2\langle x^{k+1}-z^k,x^k-\alpha_k (T(x^k)+\epsilon^k_1)-z^k\rangle+\nonumber\\
&& 2\langle x^{k+1}-z^k,\alpha_k\left[(T(x^k)+\epsilon^k_1)-(T(z^k)+\epsilon^k_2)\right]\rangle
\label{extra.cte.eq4}\\
&\le & 2\alpha_k\Vert x^{k+1}-z^k\Vert\Vert (T(z^k)+\epsilon^k_2)-(T(x^k)+\epsilon^k_1)\Vert,\nonumber
\end{eqnarray}
using the Cauchy-Schwartz inequality. Next we apply Lemma \ref{proj}(iii) 
to \eqref{algorithm.extra.cte1.F}-\eqref{algorithm.extra.cte2.F}, obtaining
\begin{eqnarray}\label{extra.cte.eq5}
\Vert x^{k+1}-z^k\Vert &=& \Vert \Pi[x^k-\alpha_k (T(z^k)+\epsilon^k_2)]-\Pi[x^k-\alpha_k(T(x^k)+\epsilon^k_1)]\Vert\nonumber\\
&\le & \alpha_k\Vert (T(z^k)+\epsilon^k_2)-(T(x^k)+\epsilon^k_1)\Vert.
\end{eqnarray}

Combining \eqref{extra.cte.eq4} and \eqref{extra.cte.eq5} we get
\begin{eqnarray}
\mathsf{I}&\le & 2\alpha_k^2\Vert (T(z^k)+\epsilon^k_2)-(T(x^k)+\epsilon^k_1)\Vert^2\nonumber\\
&\le &2\alpha_k^2\left(\Vert T(z^k)-T(x^k)\Vert+\Vert\epsilon^k_1\Vert+\Vert\epsilon^k_2\Vert\right)^2\nonumber\\
&\le &2\alpha_k^2\left(L\Vert z^k-x^k\Vert+\Vert\epsilon^k_1\Vert+\Vert\epsilon^k_2\Vert\right)^2
\label{extra.cte.eq6}\\
&\le & 6L^2\alpha_k^2\Vert z^k-x^k\Vert^2+6\alpha_k^2\Vert\epsilon^k_1\Vert^2+
6\alpha_k^2\Vert\epsilon^k_2\Vert^2,\nonumber
\end{eqnarray}
using the triangle inequality in the second inequality, Lipschitz continuity of $T$ in the third inequality and the fact 
that $(a+b+c)^2\le3a^2+3b^2+3c^2$ in the last inequality. We set $x:=x^*$ in  \eqref{extra.cte.eq3}. 
Looking now at the last term in the last equality in \eqref{extra.cte.eq3}, we get
\begin{eqnarray}
2\langle x^*-z^{k},\alpha_k\widehat{F}(\epsilon^k_2,z^k)\rangle
&=& 2\langle x^*-z^{k},\alpha_k(T(z^k)+\epsilon^k_2)\rangle\nonumber\\
&=& 2\langle x^*-z^{k},\alpha_k T(z^k)\rangle+
2\langle x^*-z^{k},\alpha_k\epsilon^k_2\rangle
\label{extra.cte.eq7}\\
&\le & 2\langle x^*-z^{k},\alpha_k\epsilon^k_2\rangle=:\mathsf{J}_k,\nonumber
\end{eqnarray}
using, in the last inequality, the fact that $\langle x^*-z^{k},\alpha_kT(z^k)\rangle\le0$, 
which follows from Assumption \ref{extra.cte.monotonicity}, the fact that $\alpha_k>0$, $x^*\in X^*$ and $z^k\in X$. Combining \eqref{extra.cte.eq3}, \eqref{extra.cte.eq6} and \eqref{extra.cte.eq7}, we get
\begin{eqnarray}	
\Vert x^{k+1}-x^*\Vert^2 & \le & \Vert x^k-x^*\Vert^2-\Vert z^k-x^k\Vert^2-\Vert z^k-x^{k+1}\Vert^2+\nonumber\\
&& 6L^2\alpha_k^2\Vert z^k-x^k\Vert^2+
6\alpha_k^2\left(\Vert\epsilon^k_1\Vert^2+\Vert\epsilon^k_2\Vert^2\right)+\mathsf{J}_k\nonumber\\
&\le & \Vert x^k-x^*\Vert^2-\rho_k\Vert z^k-x^k\Vert^2+6\alpha_k^2(\Vert\epsilon^k_2\Vert^2
+\Vert\epsilon^k_1\Vert^2)+\mathsf{J}_k,\label{extra.cte.eq8}
\end{eqnarray}
using the fact that $\rho_k=1-6L^2\alpha_{k}^2$.

Recalling that $z^k=\Pi[x^k-\alpha_k(T(x^k)+\epsilon^k_1)]$, we note that 
\begin{eqnarray}	
r_{\alpha_k}(x^k)^2&=&\Vert x^k-\Pi[x^k-\alpha_k T(x^k)]\Vert^2\nonumber\\
&\le &2\Vert x^k-z^k\Vert^2+2\Vert\Pi[x^k-\alpha_k(T(x^k)+\epsilon^k_1)]-\Pi[x^k-\alpha_k T(x^k)]\Vert^2\nonumber\\
&\le &2\Vert x^k-z^k\Vert^2+2\alpha_k^2\Vert\epsilon^k_1\Vert^2\label{extra.cte.eq9},
\end{eqnarray}
using Lemma \ref{proj}(iii) in the second inequality. From \eqref{extra.cte.eq8}, \eqref{extra.cte.eq9} and the definitions \eqref{def:A}-\eqref{def:M} 
and $\mathsf{J}_k=M_{k+1}(x^*)-M_k(x^*)$, we get the claimed relation. 
\end{proof}

We now give the proof of Lemma \ref{lema:indsums}:
\begin{proof}
We first prove the result under Assumption \ref{extra.cte.control}(i)-(ii).  Consider  item (1). 
Assume first that $m>1$ and take $i\in[m]$. For $1\leq t\leq N_i$, define $U_i^t\in\re^{n_i}$ by
$$
U_i^t:=\sum_{j=1}^t\frac{F_i(\xi_{j,i},x)-T_i(x)}{N_i}.
$$
Defining $U^0_i=0$, $\mathcal{G}_0:=\sigma(U^0_i)$ and the natural 
filtration $\mathcal{G}_t:=\sigma(\xi_{1,i},\ldots,\xi_{t,i})$ for $1\le t\le N_i$, 
$\{U^t_{i},\mathcal{G}_t\}_{t=0}^{N_i}$ defines a 
vector-valued martingale (since it is a sum of $N_i$ independent mean-zero vector random variables) whose increments satisfy
\begin{equation*}
\Lpnorm{\Vert U^t_{i}-U^{t-1}_{i}\Vert} = \Lpnorm{\frac{\Vert F_i(\xi,x)-T_i(x)\Vert}{N_i}}\leq
\frac{\Lpnorm{\|F(\xi,x)-T(x)\|}}{N_i} \leq  \frac{\sigma(x^*)\,(1+\|x-x^*\|)}{N_i},
\end{equation*}
by Assumption \ref{extra.cte.control}, using the same notation $\Vert\cdot\Vert$ for the Euclidean norm in $\re^{n_i}$ and in $\re^{n}$. Hence,
\begin{equation}\label{ee5}
\Lpnorm{\Vert U^{N_i}_{i}\Vert}\leq\frac{C_p\,\sigma(x^*)\,(1+\|x-x^*\|)}{\sqrt{N_i}},
\end{equation}
which follows from the BDG inequality \eqref{BDG}. 
For each $i\in[m]$, $\epsilon_i(x)=U^{N_i}_{i}$. Hence, since $q\ge1$, from Minkowski's inequality and \eqref{ee5}  we get:
\begin{equation}\label{ee6}
\Lpnorm{\Vert \epsilon(x)\Vert}^2=\Lqnorm{\Vert \epsilon(x)\Vert^2}\leq\sum_{i=1}^m\Lqnorm{\Vert U_{i}^{N_i}\Vert^2}
\le C_p^2\left(\sum_{i=1}^m\frac{2}{N_i}\right)\sigma(x^*)^2\,(1+\|x-x^*\|^2),
\end{equation}
using $(a+b)^2\le2a^2+2b^2$. The first claim follows from \eqref{ee6} with $\mathsf{A}=2$. If $m=1$, 
the same proof line holds with $\mathsf{A}=1$, since relation $(a+b)^2\le2a^2+2b^2$ is not required.

We now prove item (2). Suppose that $m>1$ and that $\{\xi_{j,i}:1\le i\le m,1\le j\le N_i\}$ is i.i.d.. We have
\begin{equation}\label{ee7}
\Lpnorm{\langle v,\epsilon(x)\rangle} \le\Vert v\Vert\Lpnorm{\Vert\epsilon(x)\Vert},
\end{equation}
by the Cauchy-Schwarz inequality. The claim follows from \eqref{ee6} and \eqref{ee7} with $\mathsf{B}=2$. 

Finally, we prove item (3). Suppose that $m=1$, or $m>1$ with $N_i\equiv N$, $\xi_{j,i}\equiv\xi_j$  
for all $i\in[m]$. Define $U^0:=0$ and $U^t:=(U^t_1,\ldots,U^t_m)$ for $t\ge1$ and
$
W_t:=\langle v,\cdot U^t\rangle.
$ 
Observe that $\{(W_t,\mathcal{G}_t)\}_{k=0}^N$ defines a real valued martingale with the filtration given by
$\mathcal{G}_0:=\sigma(U^0)$ and $\mathcal{G}_t:=\sigma(\xi_1,\ldots,\xi_t)$ for $t\ge1$, since it is a sum of $N$ i.i.d. random variables. 
Its increments $\Lpnorm{W_t-W_{t-1}}$ are equal to
\begin{eqnarray}
\Lpnorm{\left\langle v,\frac{F(\xi_t,x)-T(x)}{N}\right\rangle}
&\leq &\frac{\Vert v\Vert\,\Lpnorm{\|F(\xi,x)-T(x)\|}}{N}\nonumber\\
&\leq &\frac{\Vert v\Vert\sigma(x^*)(1+\|x-x^*\|)}{N},\label{ee8}
\end{eqnarray}
using the Cauchy-Schwarz inequality in the first inequality and Assumption \ref{extra.cte.control} in the last one. 
Hence, from \eqref{ee8} and the BDG-inequality \eqref{BDG}, we get 
the claim with $\mathsf{B}=1$ (in this case $\mathcal{N}=N$).

The proof of the bounds under the stronger Assumption \ref{extra.cte.control}(iii) is essentially 
the same with sharper bounds on the increments, and so we omit it.
\end{proof}

\section*{Acknowledgment} The authors are grateful for the referees' constructive comments.


\begin{thebibliography}{99}

\bibitem{bach} F. BACH AND E. MOULINES, \emph{Non-Asymptotic Analysis of Stochastic Approximation Algorithms for Machine Learning}, conference paper, Advances in Neural Information Processing Systems (NIPS), 2011.

\bibitem{bill} P. BILLINGSLEY, \emph{Convergence of Probability Measures}, John Wiley, New York, 1968.

\bibitem{iusem3} R. S. BURACHIK,A. N. IUSEM AND B. F. SVAITER, \emph{Enlargement of monotone operators with 
applications to variational inequalities}, Set-Valued Analysis, 5 (1998), pp. 159-180.

\bibitem{burk} D. L. BURKHOLDER, B. DAVIS AND R. F. GUNDY,  \emph{Integral inequalities 
for convex functions of operators on martingales}, Proceedings of the  
6th Berkeley Symposium on Mathematical Statistics and Probability, 2 (1972), pp. 223-240. 

\bibitem{nocedal} R. H. BYRD, G. M. CHIN, J. NOCEDAL AND Y. WU, \emph{Sample size selection in optimization 
methods for machine learning}, Mathematical Programming, 134 (2012), pp. 127-155.

\bibitem{chen&wets} X. CHEN, R. J-B. WETS AND Y. ZHANG, \emph{Stochastic Variational Inequalities: Residual 
Minimization Smoothing Sample Average approximations}, SIAM Journal on Optimization, 22 (2012), pp. 649-673.

\bibitem{lan} Y. CHEN, G. LAN AND Y. OUYANG, \emph{Accelerated schemes for a class of variational inequalities}, preprint. http://arxiv.org/abs/1403.4164

\bibitem{ferris} G. DENG AND M. C. FERRIS, \emph{Variable-number sample-path optimization}, Mathematical 
Programming, 117 (2009), pp. 81-109.

\bibitem{wainwright} J. C. DUCHI, P. L. BARTLETT AND M. J. WAINWRIGHT, \emph{Randomized smoothing for stochastic optimization}, SIAM Journal on Optimization, 22 (2012), pp. 674-701.

\bibitem{durrett} R. DURRET, \emph{Probability: Theory and Examples}, Cambridge University Press, Cambridge, 2010.

\bibitem{facchinei} F. FACCHINEI AND J-S. PANG, \emph{Finite-Dimensional Variational 
Inequalities and Complementarity Problems}, Springer, New York, 2003.

\bibitem{ferris:pang} M.C. FERRIS AND J.S. PANG, \emph{Engineering and economic applications of complementarity problems}, SIAM Review, 39, No. 4 (1997), pp. 669-713. 

\bibitem{friedlander} M. P. FRIEDLANDER AND G. GOH, \emph{Tail bounds for stochastic approximation}, preprint. http://arxiv.org/abs/1304.5586.

\bibitem{robinson2} G. G\"URKAN, A. Y. \"OZGE AND S. M. ROBINSON, \emph{Sample-path solution of stochastic variational inequalities}, Mathematical Programmming, 84 (1999), pp. 313-333.

\bibitem{ghadimi:lan:zhang} S. GHADIMI, G. LAN AND H. ZHANG, \emph{Mini-batch stochastic approximation methods for nonconvex stochastic composite optimization}, Mathematical Programming, ser. A, 155 (2016), pp. 267-305.

\bibitem{homem-de-mello} T. HOMEM-DE-MELLO, \emph{Variable-Sample Methods for
Stochastic Optimization}, ACM Transactions on Modeling and Computer Simulation, 13 (2003), pp. 108-133.

\bibitem{philip} A. IUSEM, A. JOFR\'E AND P. THOMPSON, \emph{Incremental constraint projection methods 
for monotone stochastic variational inequalities}, submitted.

\bibitem{Xu} H. JIANG AND H. XU, \emph{Stochastic approximation approaches to the stochastic variational 
inequality problem}, IEEE Transactions on Automatic Control, 53 (2008), pp. 1462-1475.

\bibitem{nem} A. JUDITSKY, A. NEMIROVSKI AND C. TAUVEL, \emph{Solving variational 
inequalities with stochastic mirror-prox algorithm}, Stochastic Systems, 1 (2011), pp. 17-58.

\bibitem{uday1} A. KANNAN AND U. V. SHANBHAG, \emph{Distributed computation of equilibria in monotone 
Nash games via iterative regularization techniques}, SIAM Journal on Optimization, 22 (2012), pp. 1177-1205.

\bibitem{uday3} A. KANNAN AND U. V. SHANBHAG, \emph{The pseudomonotone stochastic variational inequality problem: Analytical statements and stochastic extragradient schemes}, conference paper, American Control Conference (ACC), Portland, USA, 2014.

\bibitem{uday4} A. KANNAN AND U. V. SHANBHAG, \emph{The pseudomonotone stochastic variational inequality problem: analysis and optimal stochastic approximation schemes}, preprint. http://arxiv.org/pdf/1410.1628.pdf.

\bibitem{king:rockafellar} A.J. KING AND R.T. ROCKAFELLAR, \emph{Asymptotic theory for solutions in statistical 
estimation and stochastic programming}, Mathematics of Operations Research, 18 (1993), pp. 148-162.

\bibitem{konnov} I.V. KONNOV, \emph{Equilibrium Models and Variational Inequalities}, Elsevier, Kazan State University, Kazan, 2007.

\bibitem{korpelevich} G. M. KORPELEVICH, \emph{The extragradient method for finding saddle points and other problems}, Ekonomika i Matematcheskie Metody, 12 (1976), pp. 747-756.

\bibitem{Uday} J. KOSHAL, A. NEDI\'C AND U.V. SHANBHAG, \emph{Regularized Iterative Stochastic Approximation Methods for Stochastic Variational Inequality Problems}, IEEE Transactions on Automatic Control, 58 (2013), pp. 594-609.

\bibitem{kushner} H. J. KUSHNER AND G. G. YIN, \emph{Stochastic approximation and 
recursive algorithms and applications}, Springer, New York, 2003.

\bibitem{marinelli} C. MARINELLI AND M. R\"OCKNER, \emph{On the maximal inequalities of Burkholder, Davis and Gundy}, Expositiones Mathematicae, 34, Issue 1 (2016), pp. 1-26.

\bibitem{svaiter1} R. D. MONTEIRO AND B. F. SVAITER, \emph{On the complexity of 
the hybrid proximal extragradient method for the iterates and the ergodic mean}, SIAM Journal on Optimization, 20 (2010), pp. 275-287.
 
\bibitem{svaiter2} R. D. MONTEIRO AND B. F. SVAITER, \emph{Complexity of variants of 
Tseng's modified F-B splitting and Korpelevich's methods for hemivariational inequalities with 
applications to saddle-point and convex optimization problems}, 
SIAM Journal on Optimization, 21 (2011), pp. 1688-1720.

\bibitem{nem2} A. NEMIROVSKI, \emph{Prox-method with rate of convergence $\mathcal{O}(1/t)$ for variational inequalities with Lipschitz continuous monotone operators and smooth convex-concave saddle point problems}, SIAM Journal on  Optimization, 15 (2004), pp. 229-251.

\bibitem{shapiro2} A. NEMIROVSKI, A. JUDITSKY, G. LAN AND A. SHAPIRO, \emph{Robust stochastic approximation approach to stochastic programming}, SIAM Journal on Optimization, 19 (2009), pp. 1574-1609.

\bibitem{nesterov} Y. NESTEROV, \emph{Primal-dual subgradient methods for convex problems}, Mathematical Programming, 120 (2009), pp. 261-283.

\bibitem{ravat:shanbhag} U. RAVAT AND U.V. SHANBHAG, \emph{On the existence of solutions to stochastic quasi-variational inequality and complementarity problems}, preprint. http://arxiv.org/abs/1306.0586.

\bibitem{rob.monro} H. ROBBINS AND S. MONRO, \emph{A Stochastic Approximation Method}, The Annals of Mathematical Statistics, 22 (1951), pp. 400-407.

\bibitem{rob} H. ROBBINS AND D. O. SIEGMUND, \emph{A convergence theorem for non negative almost supermartingales and some applications}, in Optimizing methods in statistics (Proceedings of a Symposium at Ohio State University, Columbus, Ohio), J.S. Rustagi, ed., Academic Press, New York, 1971, pp. 233-257.

\bibitem{shapiro} A. SHAPIRO, D. DENTCHEVA AND A. RUSZCZYNSKI, \emph{Lectures on 
Stochastic Programming: Modeling and Theory}, SIAM, Philadelphia, 2009.

\bibitem{shanbhag:blanchet} U.V. SHANBHAG AND J. BLANCHET, \emph{Budget constrained stochastic approximation}, Proceedings of the Winter Simulation Conference, 2015.

\bibitem{wang} M. WANG AND D. BERTSEKAS, \emph{Incremental Constraint Projection Methods for 
Variational Inequalities}, Mathematical Programming Ser. A, 150 (2015), pp. 321-363. 

\bibitem{xu:zhang} H. XU and D. ZHANG, \emph{Stochastic Nash Equilibrium problems: sample average approximation and applications}, Computational Optimization and Applications, 55, Issue 3 (2013) pp. 597-645.

\bibitem{xu1} H. XU, \emph{Sample average approximation methods for a class of stochastic variational inequality problems}, Asia-Pacific Journal of Operational Research, 27, Issue 1, (2010), pp. 103-119. 

\bibitem{xu2} H. XU, \emph{Uniform exponential convergence of sample average random functions under general sampling with applications in stochastic programming}, Journal of Mathematical Analysis and Applications, 368, Issue 2 (2010), pp. 692-710.

\bibitem{uday0} F. YOUSEFIAN, A. NEDI\'C AND U. V. SHANBHAG, \emph{Self-Tuned Stochastic Approximation Schemes for Non-Lipschitzian Stochastic Multi-User Optimization and Nash Games}, IEEE Transactions on Automatic Control, 61, Issue 7, (2016), pp. 1753-1766. Extended version in: https://arxiv.org/pdf/1301.1711v1.pdf.

\bibitem{uday2} F. YOUSEFIAN, A. NEDI\'C AND U. V. SHANBHAG, \emph{Optimal robust smoothing extragradient algorithms for stochastic variational inequality problems}, IEEE Conference on Decision and Control, 2014.

\bibitem{uday5} F. YOUSEFIAN, A. NEDI\'C AND U. V. SHANBHAG, \emph{On Smoothing, Regularization and Averaging in 
Stochastic Approximation Methods for Stochastic Variational Inequalities}, preprint. http://arxiv.org/abs/1411.0209.
\end{thebibliography}
\end{document}